\documentclass[10pt,a4paper]{article}
\usepackage[english]{babel}
\usepackage[utf8]{inputenc}
\usepackage[T1]{fontenc}
\usepackage{amsmath}
\usepackage{amsfonts}
\usepackage{amssymb}
\usepackage{amsthm}
\usepackage{mathtools}
\usepackage{parskip} 
\usepackage[left=4.5cm,right=4.5cm,top=4.5cm,bottom=5.7cm]{geometry}
\usepackage{color}
\usepackage{bibgerm,cite}

\swapnumbers
\numberwithin{equation}{section}

\theoremstyle{definition}
\newtheorem{definition}{Definition}[section]
\newtheorem{example}[definition]{Example}
\newtheorem{remark}[definition]{Remark}
\newtheorem{point}[definition]{}
\theoremstyle{plain}
\newtheorem{theorem}[definition]{Theorem}
\newtheorem{proposition}[definition]{Proposition}
\newtheorem{lemma}[definition]{Lemma}
\newtheorem{corollary}[definition]{Corollary}

\newcommand{\N}{\mathbb{N}}
\newcommand{\R}{\mathbb{R}}
\newcommand{\eps}{\varepsilon}

\begin{document}

\begin{center}
\section*{Diffeomorphism groups of compact locally polyhedral manifolds}
\large{Johanna Jakob}
\end{center}
\vspace{0.5cm}
\begin{abstract}
\noindent 
A locally polyhedral manifold $M$ is a manifold which is locally diffeomorph to open subsets of convex polytopes. Manifolds with corners are a special case of local\-ly po\-ly\-hedral manifolds. We turn the group $\text{Diff}^{\infty}(M)$ of all smooth diffeo\-mor\-phisms $f\colon M\to M$ for compact smooth $M$ into a regular Lie group.
\end{abstract}

\section{Introduction and statement of results}

Manifolds of mappings are an essential concept in global analysis (cf. \cite{Eells,Kriegl}). 
It is well-known that the set $C^k(M,N)$ of all $C^k$-maps $f\colon M\to N$ can be turned into a smooth manifold for all $k\in\N_0\cup\{\infty\}$ and $\sigma$-compact, finite-dimensional smooth manifolds $M$ and $N$, where $M$ (but not $N$) may have a boundary or corners 
(cf. [6,\ 16,\ 15,\ 3,\ 18,\ 9]).
This is of particular interest for compact $M$ without boundary, because then the group $\text{Diff}^{\infty}(M)$ of all $C^{\infty}$-diffeomorphisms $f\colon M\to M$ is an open subset of $C^{\infty}(M,M)$ and a regular Lie group (see \cite{Kriegl}).
Michor \cite{Michor} was the first to discuss a smooth Lie group structure on $\text{Diff}^{\infty}(M)$ for a smooth manifold $M$ with corners, but Glöckner \cite{G} recently noticed a flaw in his construction.
However, $\text{Diff}^{\infty}(M)$ can be turned into a regular Lie group if $M$ is a compact smooth manifold with boundary (see \cite{Grong}) or a convex polytope (see \cite{G}), but for a general smooth manifold $M$ with corners, there is no valid Lie group structure on $\text{Diff}^{\infty}(M)$ available so far. This article provides a new, correct construction of a Lie group structure on $\text{Diff}^{\infty}(M)$ for compact $M$ and takes the step from manifolds with corners to the more natural, broader class of locally polyhedral manifolds.

\begin{definition}\label{cc}
Let $n\in\N_0$ and $E$ be an $n$-dimensional real vector space. Following \cite{G}, a \emph{locally polyhedral $C^{\infty}$-manifold of dimension $n$} is a Hausdorff topological space $M$, together with a set $\mathcal{A}$ of homeomorphisms $\varphi\colon U_{\varphi}\to V_{\varphi}$ from open subsets $U_{\varphi}\subseteq M$ onto open subsets of a convex polytope $P_{\varphi}\subseteq E$ with non-empty interior, such that $\mathcal{A}$ is a \emph{polyhedral smooth atlas} in the sense that $\bigcup_{\varphi\in\mathcal{A}}U_{\varphi}=M$ and $\varphi\circ\psi^{-1}$ is $C^{\infty}$ for all $\varphi,\psi\in\mathcal{A}$, and $\mathcal{A}$ is maximal among such atlases. Notably, $M$ is a $C^{\infty}$-manifold with rough boundary in the sense of \cite{GN}.
\end{definition}

Of course, each open subset $V\subseteq P$ of a convex polytope $P\subseteq E$ is a locally polyhedral $C^{\infty}$-manifold. Further examples are traditional $C^{\infty}$-manifolds with boundary or corners, since they have an atlas of charts with open images in a cube $[0,1]^n$.

For $k\in\N_0\cup\{\infty\}$, the set 
\[\text{Diff}^k(M)\coloneqq\{f\in C^k(M,M)\colon (\exists g\in C^k(M,M))\colon f\circ g=g\circ f=\text{id}_M\} \]
of all $C^k$-diffeomorphisms of $M$ is a group, where the composition of diffeomorphisms is the group multiplication and the identity map $\text{id}_M$ the neutral element. We let denote $\mathcal{V}^k(M)$ the space of all $C^k$-vector fields $X\colon M\to TM$ on $M$ and $\mathcal{V}_{\text{str}}^k(M)\subseteq \mathcal{V}^k(M)$ the subspace of all stratified $C^k$-vector fields (see Definition \ref{dd}). Our main result is the following:




\begin{theorem}\label{a}
Let $M$ be a compact locally polyhedral $C^{\infty}$-manifold. Then there is a unique smooth manifold structure on $\text{\emph{Diff}}^{\infty}(M)$ modeled on $\mathcal{V}^{\infty}_{\text{\emph{str}}}(M)$ which satisfies the following exponential law for all $l\in\N_0\cup\{\infty\}$: For each $C^l$-manifold $L$, possibly with rough boundary, a map $g\colon L\to \text{\emph{Diff}}^{\infty}(M)$ is $C^l$ if and only if
\[g^{\wedge}\colon L\times M\to M,\quad g^{\wedge}(x,y)\coloneqq g(x)(y)\]
is a $C^{l,\infty}$-map in the sense of \emph{\cite{AS}}. This manifold structure turns $\text{\emph{Diff}}^{\infty}(M)$ into a Lie group, which is $L^1$-regular in the sense of \emph{\cite{G15}} (and thus regular in the sense of \emph{\cite{Milnor}}).
\end{theorem}


This theorem is subsumed by Proposition \ref{gg} below, which contains further information. To prepare the construction of a manifold structure on $\text{Diff}^{\infty}(M)$, we need an embedding of $M$ into a manifold without boundary.


\begin{proposition}\label{b}
For every compact locally polyhedral $C^{\infty}$-manifold $M$, there exists a $\sigma$-compact finite-dimensional $C^{\infty}$-manifold $\widetilde{M}$ without boundary such that $M$ is a locally polyhedral full submanifold of $\widetilde{M}$ and the inclusion map $\text{\emph{id}}_M\colon M\hookrightarrow \widetilde{M}$ is a $C^{\infty}$-diffeomorphism onto its image.
\end{proposition}

Then $\text{Diff}^{\infty}(M)\subseteq C^{\infty}(M,\widetilde{M})$, where $\widetilde{M}$ admits a spray $X\colon T\widetilde{M}\to T(T\widetilde{M})$. The associated exponential function restricts to a local addition $\Sigma\colon D\to\widetilde{M}$ for $\widetilde{M}$ on an open subset $D\subseteq T\widetilde{M}$ (see \ref{x}).
According to \cite{Amiri},
we already have a $C^{\infty}$-manifold structure on $C^{\infty}(M,\widetilde{M})$ 
which is canonical in the sense that for each $l\in\N_0\cup\{\infty\}$ and $C^l$-manifold $L$, possibly with rough boundary, a map $g\colon L\to C^{\infty}(M,\widetilde{M})$ is $C^l$ if and only if
$g^{\wedge}\colon L\times M\to \widetilde{M},\ g^{\wedge}(x,y)\coloneqq g(x)(y)$
is a $C^{l,\infty}$-map. The topology underlying $C^{\infty}(M,\widetilde{M})$ is the compact-open $C^{\infty}$-topology.
Our goal is to prove that $\text{Diff}^{\infty}(M)$ is a submanifold of $C^{\infty}(M,\widetilde{M})$.
For this purpose, we recall the chart for $C^{\infty}(M,\widetilde{M})$ around $\text{id}_M$ 
described in \cite{Amiri}: 
Consider the open subset $D'\coloneqq (\pi_{T \widetilde{M}},\Sigma)(D)\subseteq \widetilde{M}\times \widetilde{M}$, where $\pi_{T\widetilde{M}}\colon T\widetilde{M}\to \widetilde{M}$ is the bundle projection, and the $C^{\infty}$-diffeomorphism $\theta\colon D\to D',\ v\mapsto(\pi_{T \widetilde{M}}(v),\Sigma(v))$ associated with $\Sigma$.
Then 
\[O\coloneqq\{\tau\in \mathcal{V}^{\infty}(M)\colon \tau(M)\subseteq D\}\]
is an open subset of $\mathcal{V}^{\infty}(M)$ and
\[O'\coloneqq\{g\in C^{\infty}(M,\widetilde{M})\colon (\text{id}_M,g)(M)\subseteq D'\}\]
an open subset of $C^{\infty}(M,\widetilde{M})$.
The map
\[\Psi\colon O\to O',\quad \tau\mapsto\Sigma\circ\tau\]
is a homeomorphism with inverse $\Phi\colon O'\to O,\ g\mapsto \theta^{-1}\circ(\text{id}_M,g)$ and $\Phi$ is a chart for $C^{\infty}(M,\widetilde{M})$ around $\text{id}_M$. We want to show that $\Phi$ is adapted to $\text{Diff}^{\infty}(M)$:

\begin{proposition}\label{c}
Let $n\in\N_0$ and $\widetilde{M}$ be a $\sigma$-compact $n$-dimensional $C^{\infty}$-manifold without boundary , $M\subseteq\widetilde{M}$ be a compact locally polyhedral full submanifold, $X\colon T\widetilde{M}\to T(T\widetilde{M})$ be a spray on $\widetilde{M}$ such that $X(T\partial_i M)\subseteq T(T\partial_i M)$ for all $i\in\{0,\ldots,n\}$ and $\Sigma\colon D\to \widetilde{M}$ be the associated local addition for $\widetilde{M}$.  Then there exists an open $0$-neighbourhood $Q\subseteq \mathcal{V}^{\infty}(M)$ such that $\Sigma\circ\tau$ is well-defined for all $\tau\in Q$ and
\begin{align*}
\Psi(Q\cap \mathcal{V}^{\infty}_{\text{\emph{str}}}(M))=\Psi(Q)\cap\text{\emph{Diff}}^{\infty}(M).
\end{align*}
Moreover, $\text{\emph{Diff}}^{\infty}(M)$ is a submanifold of $C^{\infty}(M,\widetilde{M})$ modeled on $\mathcal{V}^{\infty}_{\text{\emph{str}}}(M)$.
\end{proposition}

For $k\in\N$, we can construct a Banach manifold structure on $\text{Diff}^k(M)$ mo\-de\-led on $\mathcal{V}^k_{\text{str}}(M)$ analogously. This enables us to show that $\text{Diff}^{\infty}(M)=\bigcap_{k\in\N} \text{Diff}^k(M)$ is a Fr\'echet-Lie group of the class considered by Hermas and Bedida \cite{HB}, whence the $L^1$-regularity of $\text{Diff}^{\infty}(M)$ follows (see \cite{GS24}).

\section{Preliminaries and basic facts}

We write $\N\coloneqq\{1,2,\ldots\}$ and $\N_0\coloneqq\N\cup\{0\}$. 
Hausdorff locally convex topological $\R$-vector spaces will simply be called “locally convex spaces”.
We shall use the infinite-dimensional differential calculus going back to Bastiani \cite{Bastiani}, also known as Keller's $C^k_c$-theory. Therefore the manifolds we consider are modeled on locally convex spaces which may be infinite-dimensional, unless the contrary is stated. For the following approach to $C^k$-maps on non-open domains, we refer to \cite{GN}:

\begin{point}
Let $E$ and $F$ be locally convex spaces, and $U\subseteq E$ be a subset which is \emph{locally convex} (in the sense that each $x\in U$ has a convex neighbourhood in $U$) and whose interior $U^{\circ}$ is dense in $U$.
A map $f\colon U\to F$ is called $C^0$ if $f$ is continuous.
We say that $f$ is $C^1$ if $f$ is continuous, the directional derivative
\[df(x,y)\coloneqq (D_yf)(x)\coloneqq\lim_{t\to 0}\frac{f(x+ty)-f(x)}{t}\]
exists for all $(x,y)\in U^{\circ}\times E$ and admits a continuous extension 
\[df\colon U\times E\to F.\]
Given $k\in\N$, we say that $f$ is $C^k$ if $f$ is $C^1$ and $df$ is $C^{k-1}$. If $f$ is $C^k$ for all $k\in\N$, then $f$ is called $C^{\infty}$ or \emph{smooth}. 
\end{point}

\begin{point}
A subset $V\subseteq E$ of a locally convex space $E$ is called \emph{$[0,1]$-saturated}, if $tV\subseteq V$ for all $t\in[0,1]$. Each $0$-neighbourhood $U\subseteq E$ contains an open $[0,1]$-saturated $0$-neighbourhood $V$ (cf. \cite{GN}).
\end{point}



\begin{point}
Following \cite{GN}, for $k\in\N\cup\{\infty\}$, a $C^k$-\emph{ma\-ni\-fold with rough boundary} mo\-de\-led on a locally convex space $E$ is a Hausdorff topological space $M$, together with a maximal atlas $\mathcal{A}$ of $C^k$-compatible homeomorphisms $\varphi\colon U_{\varphi}\to V_{\varphi}$ from open subsets $U_{\varphi}\subseteq M$ onto locally convex subsets $V_{\varphi}\subseteq E$ with dense interior. If $k=0$, we assume in addition that $\varphi(x)\in\partial V_{\varphi}$ if and only if $\psi(x)\in \partial V_{\psi}$ for all $\varphi,\psi\in \mathcal{A}$ and $x\in U_{\varphi}\cap U_{\psi}$. If each $V_{\varphi}$ is open, $M$ is an ordinary $C^k$-manifold without boundary. Let $\pi_{TM}\colon TM\to M$ denote the bundle projection and $d\varphi\colon TU_{\varphi}\to E$ the second component of $T\varphi\colon TU_{\varphi}\to TV_{\varphi}=V_{\varphi}\times E$ for all $\varphi\in\mathcal{A}$.
\end{point}

\begin{point}
If $M$ and $N$ are $C^k$-manifolds for $k\in\N_0\cup\{\infty\}$, possibly with rough boundary, we endow the set $C^k(M,N)$ of all $C^k$-maps $f\colon M\to N$ with the \emph{compact-open $C^k$-topology} as defined \cite{GN}. An open subset $U\subseteq C^{1}(M,N)$ will often be called \emph{$C^1$-open}.
\end{point}



We shall also deal with maps on products with different degrees of differentiability in the two factors (see \cite{AS} and \cite{GN}):

\begin{point}
Let $E_1, E_2$ and $F$ be locally convex spaces, $U_1\subseteq E_1$ and $U_2\subseteq E_2$ be locally convex subsets with dense interior and $k,l\in\N_0\cup\{\infty\}$. A map $f\colon U_1\times U_2\to F$ is called $C^{k,l}$, if $d^{(0,0)}f\coloneqq f$ is continuous, the iterated directional derivatives
\begin{align*}
d^{(i,j)}f(x,y,v_1,\ldots,v_i,w_1,\ldots,w_j)\coloneqq (D_{(v_i,0)}\cdots D_{(v_1,0)}D_{(0,w_j)}\cdots D_{(0,w_1)}f)(x,y)
\end{align*}
exist for all $i,j\in\N_0$ such that $i\leq k$ and $j\leq l$, $(x,y)\in U_1^{\circ}\times U_2^{\circ}$ and $ v_1,\ldots,v_i\in E_1, w_1,\ldots, w_j\in E_2$, and admit continuous extensions
\[d^{(i,j)}f\colon U_1\times U_2\times E_1^i\times E_2^j\to F.\]
In a similar way, $C^{(\alpha_1,\ldots,\alpha_m)}$-maps $f\colon U_1\times\ldots\times U_m\to F$ can be defined for $m\in\N$ and $\alpha_1,\ldots, \alpha_m\in \N_0\cup\{\infty\}$ (see \cite{Alzaareer}).
\end{point}

\begin{point}
Let $M_1$ be a $C^k$-manifold, $M_2$ be a $C^l$-manifold and $N$ be a $C^{k+l}$-manifold, possibly with rough boundary, for $k,l\in\N_0\cup\{\infty\}$. As in \cite{GN}, we say that a map $f\colon M_1\times M_2\to N$ is $C^{k,l}$  if, for each $x=(x_1,x_2)\in M_1\times M_2$, there exist charts $\varphi_j\colon U_j\to V_j$ of $M_j$ around $x_j$ for $j\in\{1,2\}$ and a chart $\psi\colon U_{\psi}\to V_{\psi}$ of $N$ such that $f(U_1\times U_2)\subseteq U_{\psi}$ and $\psi\circ f\circ(\varphi_1^{-1}\times \varphi_2^{-1})\colon V_1\times V_2\to V_{\psi}$ is $C^{k,l}$.

Note that $C^{(\alpha_1,\ldots,\alpha_m)}$-maps $f\colon U_1\times\ldots\times U_m\to N$ can be defined analogously for $m\in\N$ and $\alpha_1,\ldots, \alpha_m\in \N_0\cup\{\infty\}$ (see \cite{GS}).
\end{point}

We get the following chain rule by applying \cite[Lemma 3.16]{Alzaareer} in local charts:

\begin{lemma}\label{bb}
Let $L_1, L_2, L_3, M_1, M_2, N$ be smooth manifolds, possibly with rough boundary, and $r,s,t\in\N_0\cup\{\infty\}$. If $f_1\colon L_1\to M_1$ is a $C^r$-map, $f_2\colon L_2\times L_3\to M_2$ a $C^{s,t}$-map and $g\colon M_1\times M_2\to N$ a $C^{r,s+t}$-map,  then $g\circ(f_1\times f_2)\colon L_1\times L_2\times L_3\to N$ is a $C^{r,s,t}$-map.
\end{lemma}


\subsubsection*{Stratified vector fields}

\begin{point}
Throughout the whole article, let $n\in\N_0$ and $E$ be an $n$-dimensional locally convex space. 
\end{point}

\begin{point}
A \emph{convex polytope} (or “polytope”, for short) in $E$ is the convex hull 
of a non-empty finite subset of $E$. An overview concerning polytopes and there faces is given in \cite{Brondsted}. 
We shall often use the fact that a polytope $P\subseteq E$ is an intersection of a finite number of closed halfspaces of $E$, i.e. there are continuous linear functionals $\lambda_1,\ldots,\lambda_m\colon E\to \R$ with $\lambda_j\neq 0$ for $j\in\{1,\ldots,m\}$ and $a_1,\ldots,a_m\in\R$ such that {$P=\bigcap_{j=1}^m \lambda_j^{-1}(]-\infty,a_j])$} (see \cite[Theorem 9.2]{Brondsted}). Given a face $F\subseteq P$, we write $\text{aff}(F)$ for the affine subspace of $E$ generated by $F$, $\text{dim}(F)\coloneqq \text{dim}(\text{aff}(F))$ and $\text{algint}(F)$ for the \emph{algebraic interior} of $F$, which is the interior of $F$ as a subset of $\text{aff}(F)$ (called the “relative interior” in \cite{Brondsted}).
If $P$ has non-empty interior in $E$, then the \emph{index} of $x$ is defined as $\text{ind}_P(x)\coloneqq n-\text{dim}\, P(x)$, where $P(x)$ is the smallest face of $P$ containing $x$. For a face $F\subseteq P$, we have $F=P(x)$ if and only if $x\in\text{algint}(F)$.
\end{point}



\begin{point}
Let $M$ be an $n$-dimensional locally polyhedral $C^{\infty}$-manifold (see De\-fi\-ni\-tion \ref{cc}) and $x\in M$. According to \cite{G}, for all charts $\varphi\colon U_{\varphi}\to V_{\varphi}\subseteq P_{\varphi}$ and $\psi\colon U_{\psi}\to V_{\psi}\subseteq P_{\psi}$ for $M$ around $x$, we have $\text{ind}_{P_{\varphi}}(\varphi(x))=\text{ind}_{P_{\psi}}(\psi(x))$, whence
\[\text{ind}_M(x)\coloneqq\text{ind}_{P_{\varphi}}(\varphi(x))\]
is a well-defined integer in $\{0,1,\ldots,n\}$. Furthermore,
\[\partial_iM\coloneqq\{x\in M\colon \text{ind}_M(x)=i\}\]
is an $(n-i)$-dimensional submanifold of $M$.
\end{point}

\begin{point}
If $P\subseteq E$ is a polytope with non-empty interior, the connected components of $\partial_iP$ are the algebraic interiors $\text{algint}(F)$ for faces $F\subseteq P$ of dimension $n-i$ (see \cite[Lemma 5.4]{G}).
\end{point}

\begin{definition}\label{dd}
Given  $k\in\N_0\cup\{\infty\}$, let $\mathcal{V}^k(M)$ denote the space of all $C^k$-vector fields $X\colon M\to TM$ on $M$.
Following \cite{G}, a $C^k$-vector field $X\in\mathcal{V}^k(M)$ is called \emph{stra\-ti\-fied} if $X(\partial_iM)\subseteq T\partial_iM$ for all $i\in \{0,\ldots,n\}$. In other words, $X$ restricts to a $C^k$-vector field on each $\partial_iM$. We write $\mathcal{V}_{\text{str}}^k(M)\subseteq \mathcal{V}^k(M)$ for the vector subspace of all stratified $C^k$-vector fields on $M$ and endow both spaces with the topology induced by $C^k(M,TM)$.
\end{definition}

\begin{point}
In the special case that $M=V\subseteq P$ is an open subset of a polytope $P\subseteq E$ with non-empty interior, we identify $TV$ with $V\times E$ as usual, and the space $C^k(V,E)$ with $\mathcal{V}^k(V)$ via $f\mapsto(\text{id}_V,f)$. Then
\[C^k_{\text{str}}(V,E)\coloneqq\{f\in C^k(V,E)\colon (\text{id}_V,f)(\partial_iV)\subseteq T\partial_iV\ \forall i\in\{0,\ldots,n\}\}\]
is the vector subspace of all stratified $C^k$-vector fields $f\colon V\to E$.
From the observations in \cite{G}, we obtain:
\end{point}

\begin{remark}\label{ii}
For a $C^k$-map $f\colon V\to E$, the following conditions are equivalent:
\begin{itemize}
\item[(a)] $f\in C^k_{\text{str}}(V,E)$;
\item[(b)] $(f+\text{id}_V)(\text{algint}(F)\cap V)\subseteq\text{aff}(F)\quad\text{for all faces }F\subseteq P.$
\end{itemize}
\end{remark}

Since $\text{aff}(F)$ is closed in $E$, it is easy to see that $C^k_{\text{str}}(V,E)$ is closed in $C^k(V,E)$, using the compact-open $C^k$-topology. 

\begin{lemma}\label{ff}
Let $M$ be a locally polyhedral $C^{\infty}$-manifold and $k\in\N_0\cup\{\infty\}$. Then the following holds:
\begin{itemize}
\item[(a)] 
If $X\colon M\to TM$ is a $C^k$-vector field, then $X\in \mathcal{V}_{\text{\emph{str}}}^k(M)$ if and only if $d\varphi\circ X\circ\varphi^{-1}\in C^k_{\text{\emph{str}}}(V_{\varphi},E)$ for each chart $\varphi\colon U_{\varphi}\to V_{\varphi}\subseteq P_{\varphi}$ of $M$.
\item[(b)]  $\mathcal{V}_{\text{\emph{str}}}^k(M)$ is closed in $\mathcal{V}^k(M)$.
\item[(c)] If $M$ is compact and $k<\infty$, then $\mathcal{V}_{\text{\emph{str}}}^k(M)$ is a Banach space.
\end{itemize}
\end{lemma}

\begin{proof}
\begin{itemize}
\item[(a)] Let $i\in\{0,\ldots,n\}$. By definition of $\partial_i M$, we have $\varphi(\partial_iM\cap U_{\varphi})=\partial_iP\cap V_{\varphi}=\partial_i V_{\varphi}$ for each chart $\varphi\colon U_{\varphi}\to V_{\varphi}\subseteq P_{\varphi}$ of $M$, and thus also $T\varphi(T(\partial_iM\cap U_{\varphi}))=T\partial_iV_{\varphi}$. As a consequence, $(T\varphi\circ X\circ\varphi^{-1})(\partial_i V_{\varphi})\subseteq T\partial_i V_{\varphi}$ for all $\varphi$ if and only if $X(\partial_iM\cap U_{\varphi})\subseteq T(\partial_iM\cap U_{\varphi})$ for all $\varphi$, which is equivalent to $X(\partial_i M)\subseteq T\partial_i M$.

\item[(b)] 
Let $(X_i)_{i\in I}$ be a net in $\mathcal{V}_{\text{str}}^k(M)$ converging to a limit $X\in \mathcal{V}^k(M)$ and $\varphi\colon U_{\varphi}\to V_{\varphi}\subseteq P_{\varphi}$ be a chart for $M$.
By \cite[Lemma 4.1.4]{GN}, the map
\[\mathcal{V}^k(M)\to C^k(V_{\varphi},E),\quad Y\mapsto d\varphi\circ Y\circ\varphi^{-1}\]
is continuous, thus $d\varphi\circ X_i\circ\varphi^{-1}\to d\varphi\circ X\circ\varphi^{-1}$. By (a), we see that $d\varphi\circ X_i\circ\varphi^{-1}\in C^k_{\text{str}}(V_{\varphi},E)$ for all $i\in I$. It follows that $d\varphi\circ X\circ\varphi^{-1}\in C^k_{\text{str}}(V_{\varphi},E)$, using that $C^k_{\text{str}}(V_{\varphi},E)$ is closed in $C^k(V_{\varphi},E)$. Applying (a) again, we deduce that $X\in\mathcal{V}_{\text{str}}^k(M)$, whence $\mathcal{V}_{\text{str}}^k(M)$ is closed.
\item[(c)] Now let $M$ be compact and $k<\infty$. Applying \cite[Proposition 4.1.28 (a) and (c)]{GN} to the bundle projection $\pi_{TM}\colon TM\to M$, we see that $\mathcal{V}^k(M)$ is a Banach space. Then $\mathcal{V}_{\text{str}}^k(M)$ is also one, being a closed vector subspace of $\mathcal{V}^k(M)$.
\end{itemize}
\end{proof}

\subsubsection*{Flows}

Now we compile some basic facts concerning flows of smooth vector fields from \cite[Section 2.5]{GN}.

\begin{point}
Let $X\colon M\to TM$ be a $C^{\infty}$-vector field on a finite-dimensional $C^{\infty}$-manifold $M$, possibly with rough boundary. Given $(t_0,y_0)\in\R\times M$,
there is an open interval $I_{t_0,y_0}\subseteq\R$ and a solution $\gamma_{t_0,y_0}\colon I_{t_0,y_0}\to M$ to the initial value problem
\begin{equation}\label{gl.m}
\dot{y}(t)=X(y(t)),\quad y(t_0)=y_0
\end{equation}
which is $\emph{maximal}$ in the sense that we have
\[I\subseteq I_{t_0,y_0}\quad\text{and}\quad \gamma=\gamma_{t_0,y_0}\vert_{I}\]
for each solution $\gamma\colon I\to M$ to \eqref{gl.m}. 
\end{point}

\begin{point}
The subset $\Omega\subseteq \R\times\R\times M$ given by
\[\Omega\coloneqq\bigcup_{(t_0,y_0)\in\R\times M}I_{t_0,y_0}\times\{(t_0,y_0)\}\]
is the domain of the \emph{maximal flow} $\text{Fl}\colon\Omega\to M,\ \text{Fl}(t,t_0,y_0)\coloneqq \gamma_{t_0,y_0}(t)$. Given $t,t_0\in\R$, we let $\Omega_{t,t_0}\coloneqq\{y_0\in M\colon t\in I_{t_0,y_0}\}$ and consider the partial map $\text{Fl}_{t,t_0}\colon\Omega_{t,t_0}\to M,\ y_0\mapsto\text{Fl}(t,t_0,y_0)$. Then we have:
\begin{itemize}
\item[(a)] $\Omega$ is open in $\R\times\R\times M$ and $\text{Fl}\colon\Omega\to M$ is smooth.
\item[(b)] $\Omega_{t,t_0}$ is open in $M$ for all $t,t_0\in\R$ and $\text{Fl}_{t,t_0}\colon\Omega_{t,t_0}\to\Omega_{t_0,t}$ is a $C^{\infty}$-diffeomorphism with $(\text{Fl}_{t,t_0})^{-1}=\text{Fl}_{t_0,t}$.
\item[(c)] If $t_0,t_1,t_2\in\R$ and $y_0\in\Omega_{t_1,t_0}$ such that $\text{Fl}_{t_1,t_0}(y_0)\in\Omega_{t_2,t_1}$, then $y_0\in\Omega_{t_2,t_0}$ and $\text{Fl}_{t_2,t_0}(y_0)=\text{Fl}_{t_2,t_1}(\text{Fl}_{t_1,t_0}(y_0))$.
\end{itemize}
\end{point}

\begin{remark}\label{r}
Let $\varphi\colon U_{\varphi}\to V_{\varphi}\subseteq E$ be a chart for $M$ and $\gamma\colon I\to M$ be a $C^{1}$-map on a non-degenerate interval $I\subseteq\R$ such that $\gamma(I)\subseteq U_{\varphi}$. Then the following conditions are equivalent:
\begin{itemize}
\item[(a)] $\gamma$ solves $\dot{y}(t)=X(y(t))$;
\item[(b)] $\varphi\circ\gamma$ solves $\dot{y}(t)=(T\varphi\circ X\circ\varphi^{-1})(y(t))$;
\item[(c)] $\varphi\circ\gamma$ solves $y'(t)=(d\varphi\circ X\circ\varphi^{-1})(y(t))$.
\end{itemize}
\end{remark}

\begin{proof}
The equivalence of (b) and (c) is clear, because $T\varphi\circ X\circ\varphi^{-1}=(\text{id}_{V_{\varphi}}, d\varphi\circ X\circ\varphi^{-1})$. Since $d\varphi(\dot{\gamma}(t))=(\varphi\circ\gamma)'(t)$ and $d\varphi(X(\gamma(t)))=(d\varphi\circ X\circ\varphi^{-1})((\varphi\circ\gamma)(t))$
for all $t\in I$, we see that $\dot{\gamma}(t)=X(\gamma(t))$ if and only if $(\varphi\circ\gamma)'(t)=(d\varphi\circ X\circ\varphi^{-1})((\varphi\circ\gamma)(t))$, which shows the equivalence of (a) and (c).
\end{proof}

\subsubsection*{Sprays}

We recall from \cite[Section 3.8]{GN} some background information concerning sprays.

\begin{point}\label{s}
Let $V\subseteq E$ be an open subset. A $C^{\infty}$-vector field $X\colon TV\to T(TV)$ is called a \emph{spray} on $V$, if it is of the form
\[X(x,v)=(x,v,v,f(x,v))\quad \text{for all } (x,v)\in TV,\]
where $f\colon TV\to E$ is a $C^{\infty}$-function satisfying $f(x,sv)=s^2f(x,v)$ for all $(x,v)\in TV$ and $s\in\R$. Then the following holds:
\begin{itemize}
\item[(a)] Given $(x_0,v_0)\in TV,$ a $C^{\infty}$-function $(\gamma,\eta)\colon I\to V\times E$ on a non-degenerate interval $I\subseteq\R$ containing $0$ is a solution to the initial value problem
\begin{equation}\label{gl.n}
\dot{y}(t)=X(y(t)),\quad y(0)=(x_0,v_0)
\end{equation}
if and only if $\eta=\gamma'$ and $\gamma\colon I\to V$ is a solution to
\begin{equation*}
y''(t)=f(y(t),y'(t)),\quad y'(0)=v_0,\quad y(0)=x_0.
\end{equation*}
\item[(b)]  Let $(\gamma,\gamma')\colon I\to V\times E$ be a solution to \eqref{gl.n}. For $s\in\R$, consider the interval $I_s\coloneqq\{t\in\R\colon st\in I\}$ and the $C^{\infty}$-function $\gamma_s\colon I_s\to V,\ \gamma_s(t)\coloneqq \gamma(st)$. Then $(\gamma_s,\gamma_s')$ is a solution to $\dot{y}(t)=X(y(t)),\ y(0)=(x_0,sv_0)$.
\end{itemize}
\end{point}

\begin{point}\label{x}
Let $M$ be a finite-dimensional $C^{\infty}$-manifold without boundary. A smooth vector field $X\colon TM\to T(TM)$ is called a \emph{spray} on $M$ if $T^2\varphi\circ X\circ T\varphi^{-1}$ is a spray on $V_{\varphi}$ for each chart $\varphi\colon U_{\varphi}\to V_{\varphi}$ of $M$. Consider the associated \emph{exponential function} 
\[\text{exp}_X\colon\Omega_{1,0}\to M,\quad v\mapsto\pi_{TM}(\text{Fl}_{1,0}(v)),\]
where $\text{Fl}_{1,0}\colon \Omega_{1,0}\to TM$ belongs to the flow of $\dot{y}(t)=X(y(t))$ and $\Omega_{1,0}\subseteq TM$ is open.
Then there is an open subset $D\subseteq \Omega_{1,0}$ such that $\Sigma\coloneqq \text{exp}_X\vert_D\colon D\to M$ is a \emph{local addition} for $M$ in the sense that
\begin{itemize}
\item[(a)] $0_x\in D$ for all $x\in M$ and $\Sigma(0_x)=x$ (for $0_x\in T_xM$); and
\item[(b)] The set $D'\coloneqq\{(\pi_{TM}(v),\Sigma(v))\colon v\in D\}$  is open in $M\times M$ and the map $\theta\colon D\to D',\ v\mapsto (\pi_{TM}(v),\Sigma(v))$ is a $C^{\infty}$-diffeomorphism.
\end{itemize}
We say that the local addition $\Sigma$ is \emph{associated with} $X$.
\end{point}

\begin{lemma}\label{o}
Let $X\colon TM\to T(TM)$ be a spray on $M$, $v_0\in M$ and $\gamma\colon I\to TM$ be a solution to the initial value problem $\dot{y}(t)=X(y(t)),\ y(0)=v_0$. Then, for each $s\in \R$, the $C^{\infty}$-function
\[\gamma_s\colon I_s\to TM,\quad \gamma_s(t)\coloneqq s\gamma(st)\]
with $I_s\coloneqq\{t\in\R\colon st\in I\}$ is a solution to $\dot{y}(t)=X(y(t)),\ y(0)=sv_0$.
\end{lemma}

\begin{proof}
Given $s\in\R$, let $t_0\in I_s$. There exists a chart $\varphi\colon U_{\varphi}\to V_{\varphi}\subseteq E$ for $M$ such that $\gamma(st_0)\in TU_{\varphi}$. Furthermore, there is a relative open interval $J\subseteq I$ with $st_0\in J$ and $\gamma(J)\subseteq TU_{\varphi}$. Considering the spray $X_{\varphi}\coloneqq (T^2{\varphi}\circ X\circ T\varphi^{-1})\colon TV_{\varphi}\to T(TV_{\varphi})$, Remark \ref{r} entails that $T\varphi\circ\gamma\colon J\to V_{\varphi}\times E$ solves $\dot{y}(t)=X_{\varphi}(y(t)).$ By \ref{s}, we have $T\varphi\circ\gamma=(\eta,\eta')$ for some $C^{\infty}$-function $\eta\colon J\to V_{\varphi}$. Writing $\eta_s\colon J_s\to V_{\varphi},\ \eta_s(t)\coloneqq\eta(st)$ with $J_s\coloneqq\{t\in\R\colon st\in J\}$, the function $(\eta_s,\eta_s')\colon J_s\to V_{\varphi}\times E$ is a solution to $\dot{y}(t)=X_{\varphi}(y(t)).$ Again with Remark \ref{r}, we see that $T\varphi^{-1}\circ (\eta_s,\eta_s')\colon J_s\to TU_{\varphi}$ solves $\dot{y}(t)=X(y(t))$. For all $t\in J_s$, we have
\begin{align*}
T\varphi^{-1}(\eta_s(t),\eta_s'(t))&=T\varphi^{-1}(\eta(st),s\eta'(st))=sT\varphi^{-1}(\eta(st),\eta'(st))\\ &=sT\varphi^{-1}(T\varphi(\gamma(st)))
=s\gamma(st)=\gamma_s(t).
\end{align*}
It follows that $\dot{\gamma_s}(t_0)=X(\gamma_s(t_0))$. In addition, $\gamma_s(0)=s\gamma(0)=sv_0$ holds.
\end{proof}

\begin{lemma}\label{p}
Let $X\colon TM\to T(TM)$ be a spray on a finite-dimensional $C^{\infty}$-manifold without boundary, 
$\text{\emph{exp}}_X\colon\Omega_{1,0}\to M$ be the associated exponential function, $v\in\Omega_{1,0}$ and $U\subseteq M$ be an open subset. Let $\gamma\colon[0,1]\to TM$ be a solution to the initial value problem 
\[\dot{y}(t)=X(y(t)),\quad y(0)=v.\]
If $sv\in\Omega_{1,0}$ and $\text{\emph{exp}}_X(sv)\in U$ for all $s\in[0,1]$, then $\gamma([0,1])\subseteq TU$.
\end{lemma}

\begin{proof}
We have \[\pi_{TM}(v)=\text{exp}_X(0_{\pi_{TM}(v)})=\text{exp}_X(0\cdot v)\in U,\] thus $\gamma(0)=v\in TU$. Now let $s\in]0,1]$. According to Lemma \ref{o}, the $C^{\infty}$-function
\[\gamma_s\colon\Big[0,\frac{1}{s}\Big]\to TM,\quad \gamma_s(t)\coloneqq s\gamma(st)\]
is a solution to
$\dot{y}(t)=X(y(t)),\ y(0)=sv.$
Consequently, \[\pi_{TM}(s\gamma(s))=\pi_{TM}(\gamma_s(1))=\text{exp}_X(sv)\in U,\] hence $s\gamma(s)\in TU$ and $\gamma(s)\in TU$.\end{proof}

\section{Proof of Proposition \ref{b}}
We will prove Proposition \ref{b} in the following way: As the first step, we construct a smooth vector field $X\colon M\to TM$ on $M$ such that $X(p)$ is a so-called strictly inner tangent vector for all $p\in M$. As the second step, we show that there \mbox{exists} $\tau>0$ such that $\text{Fl}_{\tau,0}(M)$ is a locally polyhedral full submanifold of $M\setminus\partial M$, where $\text{Fl}_{\tau,0}\colon M\to M$ belongs to the flow of $\dot{y}(t)=X(y(t))$. Hence we obtain an embedding $M\hookrightarrow M\setminus\partial M$ of $M$ into a manifold without boundary.

\begin{definition}
Let $M$ be a locally polyhedral $C^{\infty}$-manifold and $T_p M$ be the tangent space of $M$ in $p\in M$. We say that $v\in T_p M$ is an \emph{inner tangent vector}, if there exists $\varepsilon>0$ and a $C^{\infty}$-curve $\gamma: [0,\varepsilon[\to M$ such that 
\[\gamma(0)=p,\quad \dot\gamma(0)=v\quad  \text{and}\quad \gamma(]0,\varepsilon[)\subseteq M\setminus\partial M.\]
We write $T^i_p M$ for the subset of all inner tangent vectors. The elements in the interior $(T^i_p M)^{\circ}$ of $T^i_p M$ relative $T_p M$ are called \emph{strictly inner tangent vectors}. A \emph{strictly inner $C^{\infty}$-vector field} is a $C^{\infty}$-vector field $X\colon M\to TM$ such that $X(p)\in (T^i_p M)^{\circ}$ for all $p\in M$.
\end{definition}

\begin{lemma}\label{d}
Let $P\subseteq E$ be a polytope with non-empty interior, $\lambda_1,\ldots,\lambda_m\neq 0$ be continuous linear functionals on $E$ and $a_1,\ldots,a_m\in\R$ such that {$P=\bigcap_{j=1}^m \lambda_j^{-1}(]-\infty,a_j]).$} Given $x\in P$, we identify the tangent space $T_xP$ with $E$. Then, for all $v\in T_x P$, the following holds:
\begin{itemize}
\item[(a)] $v\in T^i_x P$ if and only if $\lambda_j(v)\leq 0$ for all $j\in\{1,\ldots,m\}$ with $\lambda_j(x)=a_j$.
\item[(b)] $v\in (T^i_x P)^{\circ}$ if and only if $\lambda_j(v)< 0$ for all $j\in\{1,\ldots,m\}$ with $\lambda_j(x)=a_j$.
\end{itemize}
\end{lemma}

\begin{proof}
\begin{itemize}
\item[(a)] 
We write $J\coloneqq\{j\in\{1,\ldots,m\}\colon\lambda_j(x)=a_j\}.$ Let us assume that $v\in T^i_x P$ and let $\gamma: [0,\varepsilon[\to P$ be a $C^{\infty}$-curve with $\gamma(0)=x$ and $\gamma'(0)=v$. For all $j\in J$, we have
\[\lambda_j(v)=\lambda_j(\gamma'(0))=\lambda_j\Big(\lim_{t\to 0+}\frac{1}{t}(\gamma(t)-\gamma(0))\Big)=\lim_{t\to 0+}\frac{1}{t}(\lambda_j(\gamma(t))-a_j)\leq 0.\]
Conversely, we assume that $\lambda_j(v)\leq 0$ for all $j\in J$. We choose any $y\in P^{\circ}$ and consider the smooth curve
\[\gamma\colon[0,1[\to E,\quad \gamma(t)\coloneqq t^2(y-x-v)+tv+x.\]
Then $\gamma(0)=x$ and $\gamma'(0)=v$. For all $j\in J$ and $t\in]0,1[$, we have
\begin{align*}\lambda_j(\gamma(t))&=t^2(\lambda_j(y)-\lambda_j(x)-\lambda_j(v))+t\lambda_j(v)+\lambda_j(x)\\ &=t^2\lambda_j(y)+(1-t^2)\lambda_j(x)+(t-t^2)\lambda_j(v)<a_j,
\end{align*}
because $\lambda_j(y)<a_j$ and $\lambda_j(v)\leq 0.$ For all $j\in\{1,\ldots,m\}\setminus J$, the curve $\lambda_j\circ\gamma\colon[0,1[\to\R$ is continuous with $(\lambda_j\circ\gamma)(0)<a_j$, so that there is $r_j\in]0,1[$ such that $\lambda_j(\gamma(t))<a_j$ for all $t\in[0,r_j[.$ Writing $r\coloneqq \text{min}\{r_j\colon j\in\{1,\ldots,m\}\setminus J$\}, we deduce that $\gamma(]0,r[)\subseteq P^{\circ}$, whence $v\in T_x^i P$.

\item[(b)] In (a), we have shown that $T_x^i P=\bigcap_{j\in J}\lambda_j^{-1}(]-\infty,0]).$ Since
\[(T_x^i P)^{\circ}= \bigcap_{j\in J}(\lambda_j^{-1}(]-\infty,0]))^{\circ}=\bigcap_{j\in J}\lambda_j^{-1}(]-\infty,0[),\] the assertion follows.
\end{itemize}
\end{proof}

\begin{lemma}
Let $M$ and $N$ be locally polyhedral $C^{\infty}$-manifolds and $f\colon U\to V$ be a $C^{\infty}$-diffeomorphism between open subsets $U\subseteq M$ and $V\subseteq N$. For $p\in U$, the tangent map $T_p f\colon T_pM\to T_{f(p)}N$ satisfies
\[T_pf(T_p^i M)=T_{f(p)}^i N\] and \[T_pf((T_p^i M)^{\circ})=(T_{f(p)}^i N)^{\circ}.\]
\end{lemma}

\begin{proof}
Let $v\in T_p^i M$ and $\gamma: [0,\varepsilon[\to M$ be a $C^{\infty}$-curve with $\gamma(0)=p,\ \dot\gamma(0)=v$ and $\gamma(]0,\varepsilon[)\subseteq M\setminus\partial M$. After shrinking $\eps$, we may assume that $\gamma([0,\eps[)\subseteq U$. Then $\delta\coloneqq f\circ\gamma\colon[0,\eps[\to N$ is a $C^{\infty}$-curve such that $\delta(0)=f(p),\ \dot\delta(0)=T_pf(v)$ and $\delta(]0,\eps[)\subseteq N\setminus\partial N,$ which shows that $T_pf(v)\in T_{f(p)}^i N$. Hence $T_pf(T_p^i M)\subseteq T_{f(p)}^i N$. Using this inclusion for $f^{-1}$ instead of $f$, we get $T_{f(p)}f^{-1}(T_{f(p)}^i N)\subseteq T_p^i M$, thus the first equation holds. Since $T_p f\colon T_p M\to T_{f(p)} N$ is a homeomorphism, the second equation follows.
\end{proof}

\begin{corollary}\label{e}
Let $M$ be a locally polyhedral $C^{\infty}$-manifold, $p\in M$ and $\varphi\colon U_{\varphi}\to V_{\varphi}\subseteq P_{\varphi}$ be a chart for $M$ around $p$, where $P_{\varphi}\subseteq E$ is a polytope with non-empty interior. Then we have \[T_p\varphi(T_p^i M)=T_{\varphi(p)}^i P_{\varphi}\] and \[T_p\varphi((T_p^i M)^{\circ})=(T_{\varphi(p)}^i P_{\varphi})^{\circ}.\]
\end{corollary}

\begin{definition}
A subset $C\subseteq E$ is called a \emph{convex cone}, if the following conditions hold:
\begin{itemize}
\item[(1)] $]0,\infty[C\subseteq C$;
\item[(2)] $C+C\subseteq C$;
\item[(2')] $C$ is convex.
\end{itemize}
\end{definition}

\begin{remark}
\begin{itemize}
\item[(a)] The conditions (1) and (2) are equivalent to (1) and (2').
\item[(b)] If $C\subseteq E$ is a convex cone, then also $C^{\circ}$ is one.
\end{itemize}
\end{remark}

\begin{proof}
\begin{itemize}
\item[(a)] It is clear that (1) and (2) imply (2'). To see that (2) follows from (1) and (2'), let $x,y\in C$. Then $x+y=2(\frac{1}{2}x+(1-\frac{1}{2})y)\in C$.
\item[(b)] Since $x+C^{\circ}$ is open for all $x\in C$, the set $C^{\circ}+C^{\circ}=\bigcup_{x\in C^{\circ}}(x+C^{\circ})$ is open. Together with $C^{\circ}+C^{\circ}\subseteq C$, we get $C^{\circ}+C^{\circ}\subseteq C^{\circ}$. Similarly, one can see that $]0,\infty[C^{\circ}\subseteq C^{\circ}$.
\end{itemize}
\end{proof}

\begin{lemma}
Let $M$ be a locally polyhedral $C^{\infty}$-manifold and $p\in M$. Then $T_p^i M$ (and hence also ($T_p^i M)^{\circ}$) is a convex cone in $T_p M$.
\end{lemma}

\begin{proof}
First, we assume that $M\subseteq E$ is a polytope with non-empty interior, say ${M=\bigcap_{j=1}^m \lambda_j^{-1}(]-\infty,a_j])}$ with continuous linear functionals $\lambda_1,\ldots,\lambda_m$ {$\neq 0$} on $E$ and $a_1,\ldots,a_m\in\R$.
Let $v,w\in T_p^i M$ and $t>0$. For all $j\in\{1,\ldots,m\}$ with $\lambda_j(p)=a_j$,
we have $\lambda_j(v+w)=\lambda_j(v)+\lambda_j(w)\leq 0$ and $\lambda_j(tv)=t\lambda_j(v)\leq 0$, using Lemma \ref{d} (a). Therefore $v+w$ and $tv$ are in $T_p^i M$. This shows that $T_p^i M$ is a convex cone.

Now let $M$ be any locally polyhedral $C^{\infty}$-manifold. We choose a chart $\varphi\colon U_{\varphi}\to V_{\varphi}\subseteq P_{\varphi}$ for $M$ around $p$, where $P_{\varphi}\subseteq E$ is a polytope with non-empty interior. For the addition maps $\alpha\colon T_p M\times T_p M\to T_p M$ and $\alpha_{\varphi}\colon T_{\varphi(p)} P_{\varphi}\times T_{\varphi(p)} P_{\varphi}\to T_{\varphi(p)} P_{\varphi}$, we have 
$\alpha=(T_p\varphi)^{-1}\circ \alpha_{\varphi}\circ(T_p\varphi\times T_p\varphi)$. Since $T_{\varphi(p)} P_{\varphi}$ is a convex cone as seen before, we conclude with corollary \ref{e} that
\begin{align*}
T_p^i M+T_p^i M&=\alpha(T_p^i M\times T_p^i M)=(T_p\varphi)^{-1}(\alpha_{\varphi}((T_p\varphi\times T_p\varphi)(T_p^i M\times T_p^i M)))\\
&=(T_p\varphi)^{-1}(\alpha_{\varphi}(T^i_{\varphi(p)} P_{\varphi}\times T^i_{\varphi(p)}P_{\varphi}))\subseteq (T_p\varphi)^{-1}(T^i_{\varphi(p)} P_{\varphi})= T_p^i M.
\end{align*}
Similar arguments show that $]0,\infty[T_p^i M\subseteq T_p^i M$ holds.
\end{proof}

\begin{proposition}\label{i}
Every paracompact locally polyhedral $C^{\infty}$-manifold $M$ admits a strictly inner $C^{\infty}$-vector field $X\colon M\to TM$.
\end{proposition}

\begin{proof}
Every locally polyhedral $C^{\infty}$-manifold is locally compact. According to \cite[Proposition 3.5.34]{GN}, there exists a $C^{\infty}$-partition of unity $(h_j)_{j\in J}$ on $M$ such that, for each $j\in J$, there is a chart $\varphi_j\colon U_j\to V_j\subseteq P_j$ of $M$ with $\text{supp}(h_j)\subseteq U_j$. Let $y_j\in (P_j)^{\circ}$. Lemma \ref{d} (b) helps us to see that
\[Y_j\colon V_j\to E,\quad x\mapsto y_j-x\]
is a strictly inner $C^{\infty}$-vector field on $V_j$. Thus we obtain with
\[X_j\coloneqq T\varphi_j^{-1}\circ(\text{id}_{V_j},Y_j)\circ\varphi_j\colon U_j\to TU_j\]
a strictly inner $C^{\infty}$-vector field on $U_j$. The map
\[X\colon M\to TM,\quad p\mapsto\sum_{j\in J}h_j(p)X_j(p)\]
is a finite sum for each $p\in M$ (if $p\notin U_j$, the summand should be read as 0) and a smooth vector field. Since $(T^i_pM)^{\circ}$ is a convex cone and $X_j(p)\in (T^i_pM)^{\circ}$ for all $j\in J$ with $p\in U_j$, we conclude that $X(p)\in (T^i_p M)^{\circ}$. Hence $X$ is a strictly inner $C^{\infty}$-vector field.
\end{proof}

\begin{lemma}\label{f}
Let $X \colon V\to E$ be a strictly inner $C^{\infty}$-vector field on an open subset $V\subseteq P$ of a polytope $P\subseteq E$ with non-empty interior, $\tau\geq 0$ and $\text{\emph{Fl}}_{\tau,0}^X\colon\Omega^X_{\tau,0}\to V$ be the partial map belonging to the flow of $y'(t)=X(y(t))$. If there exists $x\in \Omega^X_{\tau,0}$ such that $\text{\emph{Fl}}_{\tau,0}^X(x)\in\partial P$, then $\tau=0$.
\end{lemma}

\begin{proof}
Let $\lambda_1,\ldots,\lambda_m\neq 0$ be continuous linear functionals on $E$ and $a_1,\ldots,a_m$ {$\in\R$} such that $P=\bigcap_{j=1}^m \lambda_j^{-1}(]-\infty,a_j])$, and let $\gamma\colon I\to V$ be the maximal solution to the initial value problem
\[y'(t)=X(y(t)),\quad y(0)=x.\]
Since $\gamma(\tau)=\text{Fl}_{\tau,0}^X(x)\in\partial P$, there is $j\in\{1,\ldots,m\}$ such that $\lambda_j(\gamma(\tau))=a_j$.
Arguing by contradiction, we assume that $\tau>0$. Let $(t_n)_{n\in\N}$ be a sequence in $]-\tau,0[$ converging to $0$. Then we have $\frac{1}{t_n}(\lambda_j(\gamma(\tau+t_n))-\lambda_j(\gamma(\tau)))\geq 0$ for all $n\in\N$. It follows that
\[\lambda_j(X(\gamma(\tau)))=\lambda_j(\gamma'(\tau))=\lim_{n\to\infty}\frac{1}{t_n}(\lambda_j(\gamma(\tau+t_n))-\lambda_j(\gamma(\tau)))\geq 0,\]
whence $X(\gamma(\tau))\notin (T_{\gamma(\tau)}P)^{\circ}$, which is a contradiction to $X$ being a strictly inner vector field. Hence $\tau=0$.
\end{proof}

\begin{lemma}\label{g}
Let $X\colon M\to TM$ be a $C^{\infty}$-vector field on a finite-dimensional $C^{\infty}$-manifold $M$ with rough boundary, $\varphi\colon U_{\varphi}\to V_{\varphi}\subseteq E$ be a chart for $M$ and $X_{\varphi}\coloneqq d\varphi\circ X\circ\varphi^{-1}\colon V_{\varphi}\to E$.
Given $t,t_0\in\R$ with $t\geq t_0$, let $\text{\emph{Fl}}_{t,t_0}^{X}\colon \Omega_{t,t_0}^X\to M$ and $\text{\emph{Fl}}_{t,t_0}^{X_{\varphi}}\colon \Omega_{t,t_0}^{X_{\varphi}}\to V_{\varphi}$ be the partial maps belonging to the flows of $\dot{y}(t)=X(y(t))$ and $y'(t)=X_{\varphi}(y(t))$, respectively. Then the following holds:
\begin{itemize}
\item[(a)] If $x\in\Omega_{t,t_0}^{X_{\varphi}}$, then $\varphi^{-1}(x)\in\Omega_{t,t_0}^X$ and
\[\text{\emph{Fl}}_{t,t_0}^{X_{\varphi}}(x)=\varphi(\text{\emph{Fl}}_{t,t_0}^{X}(\varphi^{-1}(x))).\]
\item[(b)] If $y\in\Omega_{t,t_0}^X$ and $\text{\emph{Fl}}_{s,t_0}^{X}(y)\in U_{\varphi}$ for all $s\in[t_0,t]$, then $\varphi(y)\in\Omega_{t,t_0}^{X_{\varphi}}$ and
\[\text{\emph{Fl}}_{t,t_0}^{X_{\varphi}}(\varphi(y))=\varphi(\text{\emph{Fl}}_{t,t_0}^{X}(y)).\]
\end{itemize}
\end{lemma}

\begin{proof}
The assertion follows from Remark \ref{r}.
\end{proof}

\begin{remark}\label{h}
In the situation of Lemma \ref{g}, we have for all $\tau>0$:
If, for each $p\in M$, there is a chart $\varphi\colon U_{\varphi}\to V_{\varphi}$ for $M$ around $p$ such that $\varphi(p)\in\Omega_{\tau,0}^{X_{\varphi}}$, then $\Omega_{\tau,0}^X=M$.
\end{remark}

\begin{lemma}\label{k}
Let $X\colon M\to TM$ be a strictly inner $C^{\infty}$-vector field on a locally polyhedral $C^{\infty}$-manifold $M$, $\tau\geq 0$ and $\text{\emph{Fl}}_{\tau,0}^X\colon\Omega^X_{\tau,0}\to M$ be the partial map belonging to the flow of $\dot{y}(t)=X(y(t))$. If there exists $x\in \Omega^X_{\tau,0}$ such that $\text{\emph{Fl}}_{\tau,0}^X(x)\in\partial M$, then $\tau=0$.
\end{lemma}

\begin{proof}
Assuming that $\tau>0$, let us deduce that $\text{Fl}_{\tau,0}^X(x)\notin\partial M$ for each $x\in\Omega^X_{\tau,0}$. Let $\gamma\colon I\to M$ be the maximal solution to the initial value problem 
$\dot{y}(t)=X(y(t)),\ y(0)=x$
and $\varphi\colon U_{\varphi}\to V_{\varphi}\subseteq P_{\varphi}$ be a chart for $M$ around $\gamma(\tau)$. Then $\gamma^{-1}(U_{\varphi})$ is an open $\tau$-neighbourhood in $I$, so that there is $\tau'\in]0,\tau[$ with $[\tau',\tau]\subseteq\gamma^{-1}(U_{\varphi}).$
Writing $y\coloneqq \text{Fl}^X_{\tau',0}(x)$, we have $y\in\Omega^X_{\tau,\tau'}$ and $\text{Fl}^X_{s,\tau'}(y)=\text{Fl}^X_{s,0}(x)=\gamma(s)\in U_{\varphi}$ for all $s\in [\tau',\tau]$.
Lemma \ref{g} (b) entails that
\[\text{Fl}^X_{\tau,0}(x)=\text{Fl}^X_{\tau,\tau'}(y)=\varphi^{-1}\big(\text{Fl}^{X_{\varphi}}_{\tau,\tau'}(\varphi(y))\big)=\varphi^{-1}\big(\text{Fl}^{X_{\varphi}}_{\tau-\tau',0}(\varphi(y))\big),\]
where $X_{\varphi}\coloneqq d\varphi\circ X\circ\varphi^{-1}\colon V_{\varphi}\to E$ is a strictly inner vector field.
By Lemma \ref{f}, we have $\text{Fl}^{X_{\varphi}}_{\tau-\tau',0}(\varphi(y)\big)\notin\partial P_{\varphi},$ hence $\text{Fl}^X_{\tau,0}(x)\notin\partial M$.
\end{proof}

\begin{definition}
Let $\widetilde{M}$ be a finite-dimensional $C^{\infty}$-manifold modeled on $E$. We call a subset $M\subseteq \widetilde{M}$ a \emph{locally polyhedral full submanifold} if, for each $x\in M$, there exists a chart $\varphi\colon U_{\varphi}\to V_{\varphi}\subseteq E$ for $\widetilde{M}$ around $x$ such that $\varphi(U_{\varphi}\cap M)\subseteq P_{\varphi}$ is an open subset of a polytope $P_{\varphi}\subseteq E$ with non-empty interior. We say that $\varphi$ is \emph{polyhedral adapted to $M$}.
\end{definition}

Note that $M$ is a locally polyhedral $C^{\infty}$-manifold with the maximal polyhedral smooth atlas containing the charts $\varphi\vert_{U_{\varphi}\cap M}\colon U_{\varphi}\cap M\to\varphi(U_{\varphi}\cap M)$.

\begin{proposition}\label{j}
Let $X\colon M\to TM$ be a strictly inner $C^{\infty}$-vector field on a compact locally polyhedral $C^{\infty}$-manifold $M$.
Then there exists $\tau>0$ such that the partial map $\text{\emph{Fl}}_{\tau,0}^X$ belonging to the flow of $\dot{y}(t)=X(y(t))$ is defined on all of $M$ and $\text{\emph{Fl}}_{\tau,0}^X(M)$ is a locally polyhedral full submanifold of $M\setminus\partial M$.
\end{proposition}

\begin{proof}
Given $p\in  M$, we choose a chart $\varphi\coloneqq\varphi_p\colon U\to V\subseteq P$ for $M$ around $p$, where $P\subseteq E$ is a polytope with non-empty interior, say $P=\bigcap_{j=1}^m \lambda_j^{-1}(]-\infty,a_j])$ with continuous linear functionals $\lambda_1,\ldots,\lambda_m\neq 0$ on $E$ and $a_1,\ldots,a_m\in\R$.
Let $z\coloneqq\varphi(p)$, $J\coloneqq\{j\in\{1,\ldots,m\}\colon \lambda_j(z)=a_j\}$ and $W\subseteq E$ be an open subset such that $V=W\cap P$. According to Whitney's extension theorem (see \cite[Satz 3.1]{J}), the function $X_{p}\coloneqq d\varphi\circ X\circ\varphi^{-1}\colon V\to E$ has a smooth extension on $W$, i.e. a smooth function $Y\colon W\to E$ with $Y\vert_{V}=X_{p}$. Since $X(p)\in (T_p^iM)^{\circ}$ and therefore $Y(z)\in (T_z^i P)^{\circ}$, we have $\lambda_j(Y(z))<0$ for all $j\in J$. Thus there is $\theta>0$ such that $\lambda_j(Y(z))<-\theta$ for all $j\in J$. We define
\[O\coloneqq\bigg\{x\in W\colon
\begin{array}{l}
\lambda_j(Y(x))<-\theta
\\
\lambda_j(x)<a_j\quad
\end{array}
\begin{array}{l}
\text{for all }j\in J\\
\text{for all }j\in \{1,\ldots,m\}\setminus J
\end{array}
\bigg\},\]
which is an open $z$-neighbourhood in $W$. Since $P$ is locally compact, we find a compact subset $K\coloneqq K_p\subseteq V\cap O$ such that $z$ is in the interior $K^{\circ}$ of $K$ relative $P$.
The flow $\text{Fl}^{Y\vert_O}\colon\Omega^{Y\vert_O}\to O$ of the vector field $Y\vert_O\colon O\to E$ has an open
domain $\Omega^{Y\vert_O}\subseteq\R\times\R\times O$ 
and $\{0\}\times\{0\}\times K\subseteq \Omega^{Y\vert_O}$.
Using the Wallace Lemma (see \cite[A.4.3]{GN}), we find $\eps_p>0$ and an open subset $Q\subseteq O$ with $K\subseteq Q$ such that $]-\eps_p,\eps_p[\times\{0\}\times Q\subseteq\Omega^{Y\vert_O}.$ Applying the Wallace Lemma again to the relative open domain $\Omega^{X_{p}}\subseteq\R\times\R\times V$ of the flow $\text{Fl}^{X_{p}}\colon\Omega^{X_{p}}\to V$ belonging to $X_{p}$, we can, after shrinking $\eps_p$ and $Q$, achieve that $]-\eps_p,\eps_p[\times\{0\}\times (Q\cap P)\subseteq \Omega^{X_{p}}$.

Since the sets $\varphi_p^{-1}((K_p)^{\circ})$ for $p\in M$ form an open cover of $M$, there are $p_1,\ldots,p_l\in M$ such that
$M=\varphi_{p_1}^{-1}(K_{p_1})\cup\ldots\cup \varphi_{p_l}^{-1}(K_{p_l})$.
Let $\tau>0$ such that $\tau<\text{min}\{\eps_{p_1},\ldots,\eps_{p_l}\}$. Then $K_p\subseteq\Omega_{\tau,0}^{X_{p}}$ for all $p\in\{p_1,\ldots,p_l\}.$ Remark \ref{h} implies that $\text{Fl}_{\tau,0}^X$ is defined on all of $M$.

In the situation as before, let us assume that $p\in\{p_1,\ldots,p_l\}$. We define
\[Q_{\tau}\coloneqq Q_{p,\tau}\coloneqq\Big\{x\in Q\colon\lambda_j(x)<a_j+\frac{\theta}{2}\tau\quad\text{for all } j\in J\Big\},\]
which is an open subset of $Q$ with $K\subseteq Q_{\tau}$.
We want to show that \[S_{\tau}\coloneqq S_{p,\tau}\coloneqq\text{Fl}_{\tau,0}^Y (Q_{\tau})\subseteq P^{\circ}.\]
Let $x\in Q_{\tau}$. Since the curve
\begin{equation}\label{gl.a}
\gamma\colon]-\eps,\eps[\to O,\quad \gamma(t)\coloneqq\text{Fl}_{t,0}^Y(x)
\end{equation}
solves the differential equation $y'(t)=Y(y(t))$,
it follows that $(\lambda_j\circ\gamma)'(t)=\lambda_j(\gamma'(t))=\lambda_j(Y(\gamma(t)))<-\theta$ for all $j\in J$ and $t\in]-\eps,\eps[$.
Together with $\lambda_j(\gamma(0))=\lambda_j(x)<a_j+\frac{\theta}{2}\tau$, we deduce that
\[\lambda_j(\gamma(\tau))=\int_{0}^{\tau}(\lambda_j\circ\gamma)'(t)\ dt+\lambda_j(\gamma(0))<-\theta\tau+a_j+\frac{\theta}{2}\tau=a_j-\frac{\theta}{2}\tau<a_j\]
for all $j\in J$.
We also have $\lambda_j(\gamma(\tau))< a_j$ for all $j\in\{1,\ldots,m\}\setminus J$, hence $\text{Fl}_{\tau,0}^Y(x)=\gamma(\tau)\in P^{\circ}$, which proves that $S_{\tau}\subseteq P^{\circ}$.

The map $\text{Fl}_{\tau,0}^Y\colon\Omega^Y_{\tau,0}\to \Omega^Y_{0,\tau}$ being a $C^{\infty}$-diffeomorphism between open subsets of $E$, it restricts to a $C^{\infty}$-diffeomorphism $\psi\coloneqq \text{Fl}_{\tau,0}^Y\vert_{Q_{\tau}}^{S_{\tau}}\colon Q_{\tau}\to S_{\tau}$. Our goal is to show that
\begin{equation}\label{gl.b}
\psi(Q_{\tau}\cap P)=\varphi\big(\text{Fl}_{\tau,0}^X(M)\cap\varphi^{-1}(S_{\tau})\big).
\end{equation}
If $y\in \psi(Q_{\tau}\cap P)$, say $y=\text{Fl}_{\tau,0}^Y(x)$ for some $x\in Q_{\tau}\cap P$, we obtain $(\tau,0,x)\in\Omega^{X_{p}}$ and
\[y=\text{Fl}_{\tau,0}^Y(x)=\text{Fl}_{\tau,0}^{X_{p}}(x)=\varphi(\text{Fl}_{\tau,0}^X(\varphi^{-1}(x))),\]
where $\text{Fl}_{\tau,0}^X(\varphi^{-1}(x))\in\text{Fl}_{\tau,0}^X(M)\cap\varphi^{-1}(S_{\tau})$. 
If, conversely, $y\in \varphi(\text{Fl}_{\tau,0}^X(M)\cap\varphi^{-1}(S_{\tau})),$ we have $y=\text{Fl}_{\tau,0}^Y(x)$ for some $x\in Q_{\tau}$ and $y=\varphi(\text{Fl}_{\tau,0}^X(q))$ for some $q\in M$. It remains to show that $x\in P$. Let us assume that $x\notin P$ to derive a contradiction. Considering the curve $\gamma$ defined as in \eqref{gl.a}, we have $\gamma(0)=x\in E\setminus P$ and $\gamma(\tau)=y\in P^{\circ}.$ Since $\gamma([0,\tau])$ is connected, it cannot be true that $\gamma([0,\tau])\subseteq (E\setminus P)\cup P^{\circ}$. Hence there exists $t\in[0,\tau]$ such that $\gamma(t)\in\partial P$. Let $t_0$ be the maximum of such $t$. Then $\gamma([t_0,\tau])\subseteq V$ (otherwise, if there would be $s\in]t_0,\tau]$ with $\gamma(s)\notin V$ and thus $\gamma(s)\notin P$, we could find some $t\in[s,\tau]$ with $\gamma(t)\in\partial P$, a contradiction to the maximality of $t_0$).
Writing $v\coloneqq\gamma(t_0)=\text{Fl}_{t_0,0}^Y(x)\in\partial P$ and $r\coloneqq \text{Fl}_{t_0,0}^X(q)$,
we obtain
\[y=\text{Fl}^Y_{\tau,0}(x)=\text{Fl}^Y_{\tau,t_0}(v)=\text{Fl}^{X_{p}}_{\tau,t_0}(v)=\varphi(\text{Fl}^X_{\tau,t_0}(\varphi^{-1}(v)))\]
as well as
\[y=\varphi(\text{Fl}_{\tau,0}^X(q))=\varphi(\text{Fl}_{\tau,t_0}^X(r)).\]
It follows that $\varphi^{-1}(v)=r$, the maps $\varphi$ and $\text{Fl}_{\tau,t_0}^X$ being injective. Since $v\in\partial P$, we have $r\in\partial M$. But Lemma \ref{k} leads to $t_0=0$, hence $x=\text{Fl}_{0,0}^Y(x)=v\in P$, a contradiction. We deduce that $y\in \psi(Q_{\tau}\cap P)$, thus \eqref{gl.b} holds.

As a consequence, the map
\[\psi^{-1}\circ\varphi\vert_{\varphi^{-1}(S_{\tau})}\colon \varphi^{-1}(S_{\tau})\to Q_{\tau}\]
is a chart for $M\setminus\partial M$ which is polyhedral adapted to $\text{Fl}_{\tau,0}^X(M)$. 
Since \[\text{Fl}_{\tau,0}^X(M)\subseteq \varphi_{p_1}^{-1}(S_{p_1,\tau})\cup\ldots\cup \varphi_{p_l}^{-1}(S_{p_l,\tau}),\]
we conclude that $\text{Fl}_{\tau,0}^X(M)$ is a locally polyhedral full submanifold of ${M\setminus\partial M}$.
\end{proof}

\begin{lemma}\label{l}
For every compact locally polyhedral $C^{\infty}$-manifold $M$, the ma\-ni\-fold $M\setminus\partial M$ is $\sigma$-compact.
\end{lemma}

\begin{proof}
Given $p\in M$, let $\varphi_p\colon U_p\to V_p\subseteq P_p$ be a chart for $M$ around $p$, where $P_p\subseteq E$ is a polytope with non-empty interior, and $K_p$ be a compact neighbourhood of $\varphi_p(p)$ in $V_p$. After shrinking $K_p$, we may assume that $K_p$ itself is a polytope with non-empty interior, say
$K_p=\bigcap_{j=1}^m \lambda_j^{-1}(]-\infty,a_j])$ with continuous linear functionals $\lambda_1,\ldots,\lambda_m\neq 0$ on $E$ and $a_1,\ldots,a_m\in\R$. If we set $K_{p,l}\coloneqq\bigcap_{j=1}^m \lambda_j^{-1}(]-\infty,a_j-\frac{1}{l}])$ for $l\in\N$, we obtain $K_p\setminus\partial P_p=\bigcup_{l\in\N}K_{p,l}$. Since $M$ is compact, there are $p_1,\ldots,p_k\in M$ such that $M=\bigcup_{i=1}^k\varphi_{p_i}^{-1}(K_{p_i}).$ Consequently,
\[M\setminus\partial M=\bigcup_{i=1}^k\bigcup_{l\in\N}\varphi_{p_i}^{-1}(K_{p_i,l})\]
is a countable union of compact subsets.
\end{proof}

\begin{proof}[\bf Proof of Proposition \ref{b}]
By Proposition \ref{i}, there is a strictly inner $C^{\infty}$-vector field $X\colon M\to TM$. Using Proposition \ref{j}, we find $\tau>0$ such that the partial map $\text{Fl}_{\tau,0}^X$ belonging to the flow of $\dot{y}(t)=X(y(t))$ is defined on all of $M$ and $\text{Fl}_{\tau,0}^X(M)$ is a locally polyhedral full submanifold of $M\setminus\partial M$. 
We choose a bijection $f\colon M\setminus\partial M\to N$ onto a set $N$ such that $N\cap M=\emptyset$ and define $\widetilde{M}\coloneqq(N\setminus f(\text{Fl}_{\tau,0}^X(M)))\cup M$. Then the map
\begin{align*}
g\colon M\setminus\partial M\to \widetilde{M},\quad x\mapsto 
\begin{cases}
f(x), &\text{ if }x\notin \text{Fl}_{\tau,0}^X(M)\\
(\text{Fl}_{\tau,0}^X)^{-1}(x), &\text{ if }x\in \text{Fl}_{\tau,0}^X(M)
\end{cases}
\end{align*}
is a bijection. We endow $\widetilde{M}$ with the smooth manifold structure turning $g$ into a $C^{\infty}$-diffeomorphism. 
Since $\text{Fl}_{\tau,0}^X\colon M\to M\setminus\partial M$ is a $C^{\infty}$-diffeomorphism onto its image, also the inclusion map $g\circ \text{Fl}_{\tau,0}^X\colon M\to \widetilde{M},\ x\mapsto x$ is one. Moreover, $M=g(\text{Fl}_{\tau,0}^X(M))\subseteq \widetilde{M}$ is a locally polyhedral full submanifold of $\widetilde{M}$.
According to Lemma \ref{l}, $M\setminus\partial M$ is $\sigma$-compact and hence also $\widetilde{M}$.
\end{proof}

\section{The special case of a polytope}

In this chapter, we have a look on Proposition \ref{c} in the case that $M=P\subseteq E$ is a polytope with non-empty interior. 

Given $k\in\N\cup\{\infty\}$, let $\text{Diff}^k_{\text{fr}}(P)\subseteq \text{Diff}^k(P)$ be the subgroup of all $C^k$-diffeomorphisms $f\in \text{Diff}^k(P)$ which are \emph{face respecting} in the sense that $f(F)=F$ for all faces $F\subseteq P$. We endow $\text{Diff}^k(P)$ and $\text{Diff}^k_{\text{fr}}(P)$ with the topology induced by $C^k(P,E)$.

Let $\text{pr}_i\colon E\times E\to E$ be the projection onto the $i$th component for $i\in\{1,2\}$. If $\Sigma\colon D\to E$ is a local addition for $E$ and $\theta\colon D\to D',\ (x,v)\mapsto (x,\Sigma(x,v))$ the associated $C^{\infty}$-diffeomorphism (with open subsets $D\subseteq TE=E\times E$ and $D'\coloneqq (\text{pr}_1,\Sigma)(D)\subseteq E\times E$), then the subsets
\[O\coloneqq\{\tau\in C^1(P,E)\colon (\text{id}_P,\tau)(P)\subseteq D\}\] and 
\[O'\coloneqq\{g\in C^1(P,E)\colon (\text{id}_P,g)(P)\subseteq D'\}\]
are open in $C^1(P,E)$ and the map
\[\Psi\colon O\to O',\quad \tau\mapsto\Sigma\circ(\text{id}_P,\tau)\]
is a $C^{\infty}$-diffeomorphism with inverse 
$\Phi\colon O'\to O,\ g\mapsto\text{pr}_2\circ\theta^{-1}\circ(\text{id}_P,g),$
which restricts to a $C^{\infty}$-diffeomorphism
\[\Psi\vert_{O\cap C^k(P,E)}\colon O\cap C^k(P,E)\to O'\cap C^k(P,E)\]
for all $k\in\N\cup\{\infty\}.$
Our goal is to establish the following assertion:

\begin{proposition}\label{m}
Let $P\subseteq E$ be a polytope with non-empty interior, $X\colon TE\to T(TE)$ be a spray on $E$ such that $X(T\partial_iP)\subseteq T(T\partial_iP)$ for all $i\in\{0,\ldots,n\}$ and $\Sigma\colon D\to E$ be the associated local addition for $E$. Then there exists an open $0$-neighbourhood $Q\subseteq C^1(P,E)$ such that $\Sigma\circ(\text{\emph{id}}_P,\tau)$ is well-defined for all $\tau\in Q$ and
\begin{align}\label{gl.u}
\Psi(Q\cap C^k_{\text{\emph{str}}}(P,E))=\Psi(Q)\cap\text{\emph{Diff}}^k(P)\quad \text{for all }k\in\N\cup\{\infty\}.
\end{align}
Moreover, $\text{\emph{Diff}}^k(P)$ is a submanifold of $C^k(P,E)$ modeled on $C^k_{\text{\emph{str}}}(P,E)$.
\end{proposition}

For the following Lemma, we refer to \cite[Lemma 5.2. (c) and Proof 5.5]{G}.

\begin{lemma}\label{t}
Let $P\subseteq E$ be a polytope with non-empty interior.
\begin{itemize}
\item[(a)] If $f\colon U\to V$ is a $C^1$-diffeomorphism between open subsets $U,V\subseteq P$, then 
\[\text{\emph{ind}}_P(f(x))=\text{\emph{ind}}_P(x)\quad \text{for all }x\in U.\]
\item[(b)] If $f\colon P\to P$ is a $C^1$-diffeomorphism, then, for each face $F\subseteq P$, there is a face $G\subseteq P$ such that $\text{\emph{dim}}(F)=\text{\emph{dim}}(G)$ and $f(F)=G$.
\end{itemize}
\end{lemma}

This helps us to show:

\begin{lemma}\label{z}
For all $k\in\N\cup\{\infty\}$ and every polytope $P\subseteq E$ with non-empty interior, $\text{\emph{Diff}}^k_{\text{\emph{fr}}}(P)$ is an open subset of $\text{\emph{Diff}}^k(P)$.
\end{lemma}

\begin{proof} 
Let $\Vert\cdot\Vert$ be any norm on $E$. For each face $F\subseteq P$, we choose $x_F\in\text{algint}(F)$. Given two faces $F,G\subseteq P$ of same dimension with $F\neq G$, we have $r_{F,G}\coloneqq \inf_{x\in F}\Vert x-x_G\Vert>0$. Let $r$ be the minimum of all such $r_{F,G}$.
For all $f\in\text{Diff}^k_{\text{fr}}(P)$, the subset
\[B_r(f)\coloneqq\{g\in C^k(P,E)\colon\sup_{x\in P}\Vert f(x)-g(x)\Vert<r\}\]
is open in $C^k(P,E)$.
We want to show that 
\[B_r(f)\cap\text{Diff}^k(P)\subseteq\text{Diff}^k_{\text{fr}}(P),\]
which proves that $\text{Diff}^k_{\text{fr}}(P)$ is open in $\text{Diff}^k(P)$. Let $g\in B_r(f)\cap\text{Diff}^k(P)$ and $F\subseteq P$ be a face. By Lemma \ref{t} (b), there is a face $G\subseteq P$ with $\text{dim}(F)=\text{dim}(G)$ and $g(F)=G$. Let $x\in F$ such that $g(x)=x_G$. If $F\neq G$, we get
\[r_{F,G}\leq\Vert f(x)-x_G\Vert=\Vert f(x)-g(x)\Vert<r,\]
a contradiction. Hence $F=G$ and $g(F)=F$, entailing that $g\in \text{Diff}^k_{\text{fr}}(P)$.\end{proof}

\begin{lemma}\label{u}
Let $X\colon TE\to T(TE)$ be a spray and $P\subseteq E$ be a polytope with non-empty interior. Then the following are equivalent:
\begin{itemize}
\item[(a)] $X(T\partial_iP)\subseteq T(T\partial_iP)\quad \text{for all }i\in\{0,\ldots,n\};$
\item[(b)] $X(TF)\subseteq T(TF)\quad \text{for all faces } F\subseteq P$.
\end{itemize}
\end{lemma}

\begin{proof}
For a convex subset $F\subseteq E$, consider the vector subspace $E_F\coloneqq\text{aff}(F)-F$ of $E$. If we identify $TE$ with $E\times E$, we have $TF=F\times E_F$ and $T(TF)=F\times E_F\times E_F\times E_F$.

(b)$\Rightarrow$(a): Given a face $F\subseteq P$, we have
\begin{align*}
X(T\text{algint}(F))&\subseteq X(TF)\cap (\text{algint}(F)\times E^3)\subseteq T(TF)\cap (\text{algint}(F)\times E^3)\\ &=\text{algint}(F)\times (E_F)^3=T(T\text{algint}(F)),
\end{align*}
using that $E_F=E_{\text{algint}(F)}$.
Let $i\in\{0,\ldots,n\}$ and $\mathcal{F}_{n-i}(P)$ be the set of all faces of $P$ of dimension $n-i$. Since $\partial_iP=\bigcup_{F\in \mathcal{F}_{n-i}(P)}\text{algint}(F)$, it follows that
\[X(T\partial_iP)=\bigcup_{F\in \mathcal{F}_{n-i}(P)}X(T\text{algint}(F))\subseteq \bigcup_{F\in \mathcal{F}_{n-i}(P)}T(T\text{algint}(F))=T(T\partial_iP).\]
(a)$\Rightarrow$(b): Let $F\subseteq P$ be a face and $i\coloneqq n-\text{dim}(F)$. Then we have
\begin{align*}
X(T\text{algint}(F))&\subseteq X(T\partial_iP)\cap(\text{algint}(F)\times E^3) \subseteq T(T\partial_iP)\cap(\text{algint}(F)\times E^3)\\ &=\text{algint}(F)\times (E_F)^3= T(T\text{algint}(F)).
\end{align*}
Since $X$ is continuous, we obtain
\begin{align*}
X(TF)=X(\overline{T\text{algint}(F)})\subseteq\overline{X(T\text{algint}(F))}\subseteq\overline{T(T\text{algint}(F))}=T(TF).
\end{align*}
\end{proof}

\begin{example}\label{y}
Let us consider the standard spray
\[X_{\text{st}}\colon TE\to T(TE),\quad (x,v)\mapsto (x,v,v,0)\]
and any polytope $P\subseteq E$ with non-empty interior. For each face $F\subseteq P$, we have $X_{\text{st}}(TF)\subseteq F\times E_F\times E_F\times\{0\}\subseteq T(TF)$. Lemma \ref{u} shows that $X_{\text{st}}(T\partial_iP)\subseteq T(T\partial_iP)$ for all $i\in\{0,\ldots,n\}.$
\end{example}

\begin{point}
Let $P\subseteq E$ be a polytope with non-empty interior and $k\in\N\cup\{\infty\}$.
Fixing any norm on $E$ to calculate operator norms, we write
\[\Vert f\Vert_{\infty,\text{op}}\coloneqq \sup_{x\in P}\Vert f'(x)\Vert_{\text{op}}\quad\text{for all }f\in C^k(P,E).\]
One can argue similarly as in \cite{G} that the subset of all $f\in C^k(P,E)$ such that $\Vert f\Vert_{\infty,\text{op}}<1$ is open.
\end{point}

We recall a result by Glöckner (see \cite[Proposition 1.5]{G}), which is essential for the proof of Proposition \ref{m}:

\begin{proposition}\label{n}
For all $f\in C^k_{\text{\emph{str}}}(P,E)$ such that $\Vert f\Vert_{\infty,\text{\emph{op}}}<1$, we have $f+\text{\emph{id}}_P\in\text{\emph{Diff}}^k_{\text{\emph{fr}}}(P).$
\end{proposition}


\begin{proof}[\bf Proof of Proposition \ref{m}]
To prove the first assertion of Proposition \ref{m},
it suffices to show that there exists an open $0$-neighbourhood $Q\subseteq C^1(P,E)$ such that $\Sigma\circ(\text{id}_P,\tau)$ is well-defined for all $\tau\in Q$ and
\begin{align}\label{gl.v}
\Psi(Q\cap C^k_{\text{str}}(P,E))=\Psi(Q)\cap\text{Diff}^k_{\text{fr}}(P)\quad\text{for all }k\in\N\cup\{\infty\}.
\end{align}
In fact, since $\text{Diff}^1_{\text{fr}}(P)$ is open in $\text{Diff}^1(P)$ by Lemma \ref{z}, we get a $C^1$-open subset $R\subseteq C^1(P,E)$ such that $\text{Diff}^1_{\text{fr}}(P)=R\cap \text{Diff}^1(P)$; if we replace $Q$ by the $C^1$-open $0$-neighbourhood $Q\cap\Psi^{-1}(R)$,
then \eqref{gl.u} follows.

The proof is by induction on $\text{dim}(E)$. The case $\text{dim}(E)=0$ is clear, because $\tau=0$ for all $\tau\in C^k_{\text{str}}(P,E)$ and $\Sigma\circ(\text{id}_P,\tau)=\text{id}_P\in \text{Diff}^k_{\text{fr}}(P)$.
Let $\text{dim}(E)>0$
and $f\colon E\times E\to E$ be the $C^{\infty}$-map such that
\[X(z,v)=(z,v,v,f(z,v))\quad\text{for all }(z,v)\in E\times E.\]
Given a face $F\subseteq P$ with $F\neq P$, we choose any $x_F\in F$ and write $E_F\coloneqq \text{aff}(F)-x_F$ and $P_F\coloneqq F-x_F$. Then $\dim(E_F)<\dim(E)$ and $P_F$ is a polytope with non-empty interior in $E_F$. Let $\pi_F\colon E\to E_F$ be the orthogonal projection on $E_F$. Since $\pi_F$ is linear, we see that the smooth map
\[X_F\colon TE_F\to T(TE_F),\quad (z,v)\mapsto (z,v,v,\pi_F(f(z+x_F,v)))\]
is a spray on $E_F$. We want to show that
\begin{align}\label{gl.o}
X_F(T\partial_i P_F)\subseteq T(T\partial_i P_F)\quad\text{for all }i\in\{0,\ldots,\text{dim}(E_F)\}.
\end{align}
Let $G\subseteq P_F$ be a face. Then $H\coloneqq G+x_F$ is a face of $F$ and thus also a face of $P$. Lemma \ref{u} entails that $X(H\times E_H)\subseteq H\times(E_H)^3$ and hence $f(H\times E_H)\subseteq E_H$. We obtain $X_F(TG)=X_F(G\times E_H)\subseteq G\times (E_H)^3=T(TG)$. Using Lemma \ref{u} again, \eqref{gl.o} follows.

Let $\Sigma_F\colon D_F\to E_F$ be the local addition for $E_F$ associated with $X_F$.
As the next step, we show the following assertion that we need later in the proof:

\emph{If $(z,v)\in D_F$ such that $(z+x_F,v)\in D$ and $\Sigma_F(z,tv)\in P_F$ for all $t\in[0,1]$, then}
\begin{align}\label{gl.p}
\Sigma(z+x_F,v)=\Sigma_F(z,v)+x_F.
\end{align}

We have $\Sigma_F(z,v)=\gamma(1)$ for the solution $\gamma\colon[0,1]\to E_F$ to the initial value problem
\[\begin{cases}
y''(t)&=\pi_F(f(y(t)+x_F,y'(t)))\\
y'(0)&=v\\
y(0)&=z.
\end{cases}\]
Since $\gamma(t)=\Sigma_F(z,tv)\in P_F$ for all $t\in[0,1]$, it follows that $(\gamma(t)+x_F,\gamma'(t))\in F\times E_F$, thus $f(\gamma(t)+x_F,\gamma'(t))\in E_F$ and $\pi_F(f(\gamma(t)+x_F,\gamma'(t)))=f(\gamma(t)+x_F,\gamma'(t))$.
For the smooth curve 
\[\delta\colon[0,1]\to E,\quad \delta(t)\coloneqq \gamma(t)+x_F,\]
we deduce that 
\[\delta''(t)=\gamma''(t)=f((\gamma(t)+x_F,\gamma'(t))=f(\delta(t),\delta'(t)),\]
whence $\delta$ solves the initial value problem
\[\begin{cases}
y''(t)&=f(y(t),y'(t))\\
y'(0)&=v\\
y(0)&=z+x_F.
\end{cases}\]
Consequently,
\begin{align*}
\Sigma(z+x_F,v)&=\delta(1)=\gamma(1)+x_F=\Sigma_F(z,v)+x_F,
\end{align*}
which was to show.

Now we apply the induction hypothesis to find an open $0$-neighbourhood $Q_F\subseteq C^1(P_F,E_F)$ such that $\Sigma_F\circ(\text{id}_{P_F},\sigma)$ is well-defined for all $\sigma\in Q_F$ and the map
\[\Psi_F\colon Q_F\to C^1(P_F,E_F),\quad \sigma\mapsto\Sigma_F\circ(\text{id}_{P_F},\sigma)\]
satisfies
\begin{equation}\label{gl.e}
\Psi_F(Q_F\cap C^1_{\text{str}}(P_F,E_F))=\Psi_F(Q_F)\cap\text{Diff}^1_{\text{fr}}(P_F).
\end{equation}
Consider the translation map $\alpha_{x_F}\colon P_F\to F,\ z\mapsto z+x_F$ and its inverse
$\alpha_{-x_F}\colon F\to P_F,\ z\mapsto z-x_F$.
Since the map
\[C^1(P_F,E_F)\to C^1(F,E\times E),\quad\sigma\mapsto (\text{id}_F,\sigma\circ \alpha_{-x_F})\]
is continuous, after shrinking $Q_F$, we may assume that 
\begin{align}\label{gl.q}
(\text{id}_F,\sigma\circ \alpha_{-x_F})(F)\subseteq D
\end{align}
for all $\sigma\in Q_F$. After replacing $Q_F$ by a $[0,1]$-saturated $C^1$-open $0$-neigh\-bour\-hood contained in $Q_F$, we may assume in addition that $Q_F$ is $[0,1]$-saturated. We define the $C^1$-open $0$-neighbourhood
\[Q\coloneqq\{\tau\in C^1(P,E)\colon (\text{id}_P,\tau)(P)\subseteq D\text{ and } \Vert \Sigma\circ(\text{id}_P,\tau)-\text{id}_P\Vert_{\infty,\text{op}}<1\}.\]
The map
\[\rho_1\colon C^1_{\text{str}}(P,E)\to C^1_{\text{str}}(P_F,E_F),\quad \tau\mapsto\tau\vert_F \circ\alpha_{x_F}\]
being continuous, we find an open $0$-neighbourhood $V_F\subseteq C^1(P,E)$ such that 
\[V_F\cap C^1_{\text{str}}(P,E)=\rho_1^{-1}(Q_F\cap C^1_{\text{str}}(P_F,E_F)).\]
Since $\Psi_F(Q_F)$ is open in $C^1(P_F,E_F)$ and the map
\[\rho_2\colon \text{Diff}^1_{\text{fr}}(P)\to \text{Diff}^1_{\text{fr}}(P_F),\quad \varphi\mapsto\alpha_{-x_F}\circ\varphi\vert_{F} \circ \alpha_{x_F}\]
is continuous, there is an open $\text{id}_P$-neighbourhood $W_F\subseteq C^1(P,E)$ such that
\[W_F\cap \text{Diff}^1_{\text{fr}}(P)=\rho_2^{-1}(\Psi_F(Q_F)\cap \text{Diff}^1_{\text{fr}}(P_F)).\]
Now we replace $Q$ by the $C^1$-open $0$-neighbourhood \[Q\cap\bigcap_{F}(V_F\cap\Psi^{-1}(W_F)),\]
where $F$ ranges through all faces $F\subseteq P$ with $F\neq P$.

In the following, we show that each given $\tau\in Q\cap C^1_{\text{str}}(P,E)$ satisfies 
\[(\Sigma\circ(\text{id}_P,\tau))(F)=F\]
for all faces $F\subseteq P$ with $F\neq P$. Since $\tau\in V_F\cap C^1_{\text{str}}(P,E)$, we have $\tau\circ\alpha_{x_F}\in Q_F\cap C^1_{\text{str}}(P_F,E_F)$. The latter set being $[0,1]$-saturated, we see that $t(\tau\circ\alpha_{x_F})\in Q_F\cap C^1_{\text{str}}(P_F,E_F)$ for all $t\in[0,1]$, hence ${\Sigma_F\circ(\text{id}_{P_F},t(\tau\circ\alpha_{x_F}))\in\text{Diff}^1_{\text{fr}}(P_F)}$ by \eqref{gl.e}, and in particular $\Sigma_F(z-x_F,t\tau(z))\in P_F$ for all $z\in F$.
Using \eqref{gl.p}, we obtain
\[\Sigma(z,\tau(z))=\Sigma_F(z-x_F,\tau(z))+x_F=(\Sigma_F\circ(\text{id}_{P_F},\tau\circ\alpha_{x_F}))(z-x_F)+x_F\]
for all $z\in F$. It follows that
\[(\Sigma\circ(\text{id}_P,\tau))(F)=(\Sigma_F\circ(\text{id}_{P_F},\tau\circ\alpha_{x_F}))(P_F)+x_F=P_F+x_F=F.\]

Now let $k\in\N\cup\{\infty\}$ and $\tau\in Q\cap C^k_{\text{str}}(P,E)$. The preceding consideration shows that
$(\Sigma\circ(\text{id}_P,\tau))-\text{id}_P$ is stratified, thus 
$(\Sigma\circ(\text{id}_P,\tau))-\text{id}_P\in C^k_{\text{str}}(P,E)$.
Together with $\Vert\Sigma\circ(\text{id}_P,\tau)-\text{id}_P\Vert_{\infty,\text{op}}<1$, Proposition \ref{n} leads to $\Sigma\circ(\text{id}_P,\tau)\in\text{Diff}^k_{\text{fr}}(P)$. Hence 
$\Psi(Q\cap C^k_{\text{str}}(P,E))\subseteq\Psi(Q)\cap\text{Diff}^k_{\text{fr}}(P)$.

To prove the inverse inclusion, let $\tau\in Q$ such that $\Psi(\tau)=\Sigma\circ(\text{id}_P,\tau)\in \text{Diff}^k_{\text{fr}}(P)$.
Then $\tau\in C^k(P,E)$, and it remains to show that $\tau$ is stratified. Let $F\subseteq P$ be a face with $F\neq P$. Since $\Sigma\circ(\text{id}_P,\tau)\in W_F\cap \text{Diff}^1_{\text{fr}}(P)$, we have $\alpha_{-x_F}\circ(\Sigma\circ(\text{id}_P,\tau))\circ\alpha_{x_F}\in\Psi_F(Q_F)\cap\text{Diff}^1_{\text{fr}}(P_F)$.
According to \eqref{gl.e}, there is $\sigma\in Q_F\cap C^1_{\text{str}}(P_F,E_F)$ such that \[(\Sigma\circ(\text{id}_P,\tau))\vert_F=\alpha_{x_F}\circ (\Sigma_F\circ(\text{id}_{P_F},\sigma))\circ\alpha_{-x_F}.\]
Note that $(z,\sigma(z-x_F))\in D$ for all $z\in F$ by \eqref{gl.q}, which allows us to apply \eqref{gl.p}, so that we obtain
\[\Sigma(z,\sigma(z-x_F))=\Sigma_F(z-x_F,\sigma(z-x_F))+x_F=\Sigma(z,\tau(z)).\]
Hence $\tau(z)=\sigma(z-x_F)\in E_F$, using that the map $D\to E\times E, (u,v)\mapsto(u,\Sigma(u,v))$ is injective. We conclude that $\tau\in C^k_{\text{str}}(P,E)$, so that we have established \eqref{gl.v}.

At least, we verify that $\text{Diff}^k(P)$ is a submanifold of $C^k(P,E)$ modeled on $C^k_{\text{str}}(P,E)$ for all $k\in\N\cup\{\infty\}$. Note that a spray on $E$ as in Proposition \ref{m} always exists (e.g. the standard spray on $E$). Let $Q\subseteq C^1(P,E)$ be a $C^1$-open $0$-neighbourhood such that $\Sigma\circ(\text{id}_P,\tau)$ is well-defined for all $\tau\in Q$ and \eqref{gl.u} holds. Given $g\in \text{Diff}^k(P),$
the map
\[R_g\colon C^1(P,E)\to C^1(P,E),\quad f\mapsto f\circ g\]
is a $C^{\infty}$-diffeomorphism which restricts to a $C^{\infty}$-diffeomorphism $R_g\colon C^k(P,E)\to C^k(P,E)$ 
such that $R_g(\text{Diff}^k(P))=\text{Diff}^k(P)$ and $R_g(\text{id}_P)=g$. Hence
\[R_g\circ\Psi\vert_{Q\cap C^k(P,E)}\colon Q\cap C^k(P,E)\to (R_g\circ\Psi)(Q)\cap C^k(P,E)\]
is a $C^{\infty}$-diffeomorphism such that
\begin{align*}
(R_g\circ\Psi)(Q\cap C^k_{\text{str}}(P,E))&=R_g(\Psi(Q)\cap\text{Diff}^k(P))
=(R_g\circ \Psi)(Q)\cap\text{Diff}^k(P),
\end{align*}
showing that $\big(R_g\circ\Psi\vert_{Q\cap C^k(P,E)}\big)^{-1}$
is a chart for $C^k(P,E)$ around $g$ adapted to $\text{Diff}^k(P)$. This completes the proof.
\end{proof}



\begin{remark}\label{v}
For a polytope $P\subseteq E$ with non-empty interior, the submanifold structure on  $\text{Diff}^{\infty}(P)$ induced by $C^{\infty}(P,E)$ coincides with the Lie group structure on $\text{Diff}^{\infty}(P)$ constructed in \cite{G}.
\end{remark}

\begin{proof}
Let $\text{Diff}^{\infty}_{\text{fr}}(P)$ be endowed with the submanifold structure induced by $C^{\infty}(P,E)$. 
The Lie group structure on $\text{Diff}^{\infty}(P)$ constructed in \cite{G} has the following two properties:
\begin{itemize}
\item[(a)] The inclusion map $\text{Diff}^{\infty}_{\text{fr}}(P)\hookrightarrow \text{Diff}^{\infty}(P)$ is a  $C^{\infty}$-diffeomorphism onto an open subset of $\text{Diff}^{\infty}(P)$.
\item[(b)] The  map 
\[R_g\colon \text{Diff}^{\infty}(P)\to \text{Diff}^{\infty}(P),\quad f\mapsto f\circ g\]
is a $C^{\infty}$-diffeomorphism for all $g\in \text{Diff}^{\infty}(P)$.
\end{itemize}
But a manifold structure on $\text{Diff}^{\infty}(P)$ is uniquely determined by (a) and (b). Since the submanifold structure on  $\text{Diff}^{\infty}(P)$ induced by $C^{\infty}(P,E)$ also satisfies (a) and (b), it coincides with the Lie group structure in \cite{G}.
\end{proof}


\section{Proof of Theorem \ref{a} and Proposition \ref{c}}

The goal of this chapter is to prove the following two Propositions \ref{gg} and \ref{hh}, which subsume Theorem \ref{a} and Proposition \ref{c}.

\begin{proposition}\label{gg}
Let $M$ be a compact locally polyhedral $C^{\infty}$-manifold and $k\in\N\cup\{\infty\}$. Then the following holds: 
\begin{itemize}
\item[(a)] There is a unique smooth manifold structure on $\text{\emph{Diff}}^k(M)$ modeled on $\mathcal{V}^k_{\text{\emph{str}}}(M)$ which satisfies the following exponential law for all $l\in\N_0\cup\{\infty\}$: For each $C^l$-manifold $L$, possibly with rough boundary, a map $g\colon L\to \text{\emph{Diff}}^k(M)$ is $C^l$ if and only if
\[g^{\wedge}\colon L\times M\to M,\quad g^{\wedge}(x,y)\coloneqq g(x)(y)\]
is a $C^{l,k}$-map.
\item[(b)] Even a more general version of the exponential law is valid: Let $m\in\N$ and $\alpha_1,\ldots,\alpha_m\in \N_0\cup\{\infty\}$.
If $L_i$ is a $C^{\alpha_i}$-manifold for $i\in\{1,\ldots,m\}$, possibly with rough boundary, a map $g\colon L_1\times\ldots\times L_m\to \text{\emph{Diff}}^k(M)$ is $C^{(\alpha_1,\ldots,\alpha_m)}$ if and only if
\[g^{\wedge}\colon L_1\times\ldots\times L_m\times M\to M,\quad g^{\wedge}(x_1,\ldots,x_m,y)\coloneqq g(x_1,\ldots,x_m)(y)\]
is a $C^{(\alpha_1,\ldots,\alpha_m,k)}$-map.
\item[(c)] For all $l\in\N_0\cup\{\infty\}$, the composition map
\[c_{k,l}\colon\text{\emph{Diff}}^{k+l}(M)\times\text{\emph{Diff}}^k(M)\to\text{\emph{Diff}}^k(M),\quad (f,g)\mapsto f\circ g\]
is $C^{\infty,l}$ (and thus $C^l$);
moreover, the inversion map
\[\iota_{k,l}\colon\text{\emph{Diff}}^{k+l}(M)\to\text{\emph{Diff}}^k(M),\quad f\mapsto f^{-1}\]
is $C^l$.
\item[(d)] $\text{\emph{Diff}}^k(M)$ is a topological group.
\item[(e)] $\text{\emph{Diff}}^{\infty}(M)$ is a Lie group, which is $L^1$-regular in the sense of \emph{\cite{G15}} (and thus regular in the sense of \emph{\cite{Milnor}}).
\end{itemize}
\end{proposition}

The main idea of the proof is the following. Let $\widetilde{M}$ be a $\sigma$-compact finite-dimensional $C^{\infty}$-manifold without boundary such that $M$ is a locally polyhedral full submanifold of $\widetilde{M}$ and $\text{id}_M\colon M\to \widetilde{M}$ is a $C^{\infty}$-diffeomorphism onto its image. We endow $C^k (M,\widetilde{M})$ with the canonical ma\-ni\-fold structure as in \cite{Amiri} and have to show that $\text{Diff}^k(M)\subseteq C^k (M,\widetilde{M})$ is a submanifold for all $k\in\N\cup\{\infty\}$.
Therefore we consider the chart for $C^k(M,\widetilde{M})$ around $\text{id}_M$ defined in \cite{Amiri}: Let $\Sigma\colon D\to \widetilde{M}$ be a local addition for $\widetilde{M}$ and $\theta\colon D\to D',\ v\mapsto(\pi_{T \widetilde{M}}(v),\Sigma(v))$ be the $C^{\infty}$-diffeomorphism associated with $\Sigma$ (with open subsets $D\subseteq T\widetilde{M}$ and $D'\coloneqq (\pi_{T \widetilde{M}},\Sigma)(D)\subseteq \widetilde{M}\times \widetilde{M}$).
Then 
\[O\coloneqq\{\tau\in \mathcal{V}^1(M)\colon \tau(M)\subseteq D\}\]
is an open subset of $\mathcal{V}^1(M)$ and 
\[O'\coloneqq\{g\in C^1(M,\widetilde{M})\colon (\text{id}_M,g)(M)\subseteq D'\}\]
an open subset of $C^1(M,\widetilde{M})$.
The map
\[\Psi\colon O\to O',\quad \tau\mapsto\Sigma\circ\tau\]
is a homeomorphism with inverse $\Phi\colon O'\to O,\ g\mapsto \theta^{-1}\circ(\text{id}_M,g).$ For all $k\in\N\cup\{\infty\}$, the restriction
\[\Phi_k\coloneqq\Phi\vert_{O'\cap C^k(M,\widetilde{M})}\colon O'\cap C^k(M,\widetilde{M})\to O\cap \mathcal{V}^k(M)\]
is a chart for
$C^k(M,\widetilde{M})$ around $\text{id}_M$. The next proposition 
shows that $\Phi_k$ is adapted to $\text{Diff}^k(M)$ for a suitable choice of $\Sigma$:

\begin{proposition}\label{hh}
Let $\widetilde{M}$ be a $\sigma$-compact $n$-dimensional $C^{\infty}$-manifold without boundary , $M\subseteq\widetilde{M}$ be a compact locally polyhedral full submanifold, $X\colon T\widetilde{M}\to T(T\widetilde{M})$ be a spray on $\widetilde{M}$ such that $X(T\partial_i M)\subseteq T(T\partial_i M)$ for all $i\in\{0,\ldots,n\}$ and $\Sigma\colon D\to \widetilde{M}$ be the associated local addition for $\widetilde{M}$.  Then there exists an open $0$-neighbourhood $Q\subseteq \mathcal{V}^1(M)$ such that $\Sigma\circ\tau$ is well-defined for all $\tau\in Q$ and
\begin{align}\label{gl.r}
\Psi(Q\cap \mathcal{V}^k_{\text{\emph{str}}}(M))=\Psi(Q)\cap\text{\emph{Diff}}^k(M)\quad\text{for all }k\in\N\cup\{\infty\}.
\end{align}
Moreover, $\text{\emph{Diff}}^k(M)$ is a submanifold of $C^k(M,\widetilde{M})$ modeled on $\mathcal{V}^k_{\text{\emph{str}}}(M)$.
\end{proposition}

The proof of Proposition \ref{hh} requires some preparations. First, we construct an appropriate spray on $\widetilde{M}$:

\begin{lemma}\label{w}
Let $\widetilde{M}$ be a $\sigma$-compact $n$-dimensional $C^{\infty}$-manifold without boundary and $M\subseteq \widetilde{M}$ be a locally polyhedral full submanifold. Then there exists a spray $X\colon T\widetilde{M}\to T(T\widetilde{M})$ on $\widetilde{M}$ such that 
\[X(T\partial_i M)\subseteq T(T\partial_i M)\quad\text{for all }i\in\{0,\ldots,n\}.\]
\end{lemma}

\begin{proof}
For each $x\in M$, there exists a chart $\varphi_x\colon\widetilde{U_x}\to \widetilde{V_x}\subseteq E$ of $\widetilde{M}$ around $x$ polyhedral adapted to $M$ and a compact $x$-neighbourhood $K_x\subseteq \widetilde{U_x}$. Since $M$ is compact, we find $x_1,\ldots,x_m\in M$ such that $M\subseteq K_{x_1}\cup\ldots\cup K_{x_m}.$ For each $x\in U\coloneqq \widetilde{M}\setminus (K_{x_1}\cup\ldots\cup K_{x_m})$, we have a chart $\varphi_x\colon \widetilde{U_x}\to \widetilde{V_x}$ of $\widetilde{M}$ around $x$ with $\widetilde{U_x}\subseteq U$. Let $(h_j)_{j\in J}$ be a partition of unity on $\widetilde{M}$ subordinated to $\{\widetilde{U_x}\colon x\in\{x_1,\ldots,x_m\}\cup U\}.$ Then, for all $j\in J$, there is a chart $\varphi_j\colon \widetilde{U_j}\to \widetilde{V_j}$ with $\text{supp}(h_j)\subseteq \widetilde{U_j}$ and $\varphi_j\in\{\varphi_x\colon x\in\{x_1,\ldots,x_m\}\cup U\}.$ 
Using the standard spray $X_{\text{st}}\colon TE\to T(TE),\ (x,v)\mapsto (x,v,v,0)$,
we define a spray $X_j\coloneqq T^2\varphi_j^{-1}\circ X_{\text{st}}\vert_{T\widetilde{V_j}}\circ T\varphi_j\colon T\widetilde{U_j}\to T(T\widetilde{U_j})$. Then
\[X\colon T\widetilde{M}\to T(T\widetilde{M}),\quad v\mapsto \sum_{j\in J}h_j(\pi_{T\widetilde{M}}(v))X_j(v)\]
is a finite sum for each $v\in T\widetilde{M}$ (if $v\notin T\widetilde{U_j}$, the summand should be read as $0$) and a spray on $\widetilde{M}$ (cf. \cite[p$.\, 320$]{GN}). 
It remains to show that $X(T\partial_i M)\subseteq T(T\partial_i M)$ for all $i\in\{0,\ldots,n\}$.
Let $v\in T\partial_i M$. For each $j\in J$ with $h_j(\pi_{T\widetilde{M}}(v))\neq 0$, there is a polytope $P_j\subseteq E$ with non-empty interior such that $V_j\coloneqq\varphi_j(\widetilde{U_j}\cap M)$ is an open subset of $P_j$. Since $\varphi_j(\widetilde{U_j}\cap \partial_i M)=V_j\cap\partial_iP$ and $v\in T(\widetilde{U_j}\cap\partial_i M)$, it follows that $T\varphi_j(v)\in T(V_j\cap\partial_iP)$. Together with $X_{\text{st}}(T\partial_iP)\subseteq T(T\partial_iP)$, this implies $X_{\text{st}}(T\varphi_j(v))\in T(T(V_j\cap\partial_iP))$ and hence $X_j(v)\in T(T(\widetilde{U_j}\cap\partial_iM))\subseteq T(T\partial_iM)$. We conclude that $X(v)\in T(T\partial_iM)$.
\end{proof}



\begin{lemma}\label{q}
Every compact locally polyhedral $C^{\infty}$-manifold $M$ is metrizable.
\end{lemma}

\begin{proof}
By Proposition \ref{b}, there is an embedding of manifolds $M\hookrightarrow\widetilde{M}$ into a $\sigma$-compact finite-dimensional $C^{\infty}$-manifold $\widetilde{M}$. Whitney's Embedding Theorem provides an embedding $\widetilde{M}\hookrightarrow\R^m$ for some $m\in\N$.  Consequently, $M$ embeds in $\R^m$, whence $M$ is metrizable. 
\end{proof}

\begin{lemma}\label{aa}
Let $P\subseteq E$ be a polytope with non-empty interior. For all $k\in\N\cup\{\infty\}$ and $l\in\N_0\cup\{\infty\}$, the map
\[\eta_{k,l}\colon\text{\emph{Diff}}^{k+l}(P)\to\text{\emph{Diff}}^k(P),\quad f\mapsto f^{-1}\]
is $C^l$.
\end{lemma}

\begin{proof}
We follow the idea in \cite[pp$.\, 9$f]{G}. Since the evaluation map $\eps_{k+l}\colon \text{Diff}^{k+l}(P)\times P\to P$ is $C^{\infty,k+l}$ and thus $C^{k+l}$, 
the map
\[\eps_{k+l}\circ(\varphi^{-1}\times\text{id}_P)\colon V_{\varphi}\times P\to P,\quad (h,x)\mapsto\varphi^{-1}(h)(x)\]
is $C^{k+l}$ for each chart $\varphi\colon U_{\varphi}\to V_{\varphi}\subseteq C^{k+l}_{\text{str}}(P,E)$ of $\text{Diff}^{k+l}(P)$. By \cite[Lemma 4.2]{G}, also 
\[g\colon V_{\varphi}\times P\to P,\quad (h,x)\mapsto(\varphi^{-1}(h))^{-1}(x)\]
is $C^{k+l}$, whence
\[g\circ(\varphi\times \text{id}_P)\colon U_{\varphi}\times P\to P,\quad (f,x)\mapsto f^{-1}(x)\]
is $C^{k+l}$. We deduce that
\[\eta_{k,l}^{\wedge}\colon \text{Diff}^{k+l}(P)\times P\to P,\quad (f,x)\mapsto f^{-1}(x)\]
is $C^{k+l}$ and thus $C^{l,k}$, such that the exponential law entails that $\eta_{k,l}$ is $C^l$.
\end{proof}

\begin{proof}[\bf Proof of Propositions \ref{gg} and \ref{hh}] 
Let $M$ be a compact locally polyhedral $C^{\infty}$-manifold. By Proposition \ref{b}, there exists a $\sigma$-compact finite-dimensional $C^{\infty}$-manifold $\widetilde{M}$ without boundary such that $M$ is a locally polyhedral full submanifold of $\widetilde{M}$ and $\text{id}_M\colon M\to\widetilde{M}$ is a $C^{\infty}$-diffeomorphism onto its image. Lemma \ref{w} provides a spray $X\colon T\widetilde{M}\to T(T\widetilde{M})$ on $\widetilde{M}$ such that 
$X(T\partial_i M)\subseteq T(T\partial_i M)$ for all $i\in\{0,\ldots,n\}.$
Let $\Sigma\colon D\to\widetilde{M}$ be the associated local addition for $\widetilde{M}$. At the beginning, we prepare the construction of the $0$-neighbourhood $Q\subseteq\mathcal{V}^1(M)$ which is asked for in Proposition \ref{hh}.
Given $j\in M$, there is a chart $\varphi_j\colon  \widetilde{U_j}\to \widetilde{V_j}\subseteq E$ for $\widetilde{M}$ around $j$ polyhedral adapted to $M$. Let $\varphi_j\vert_{U_j}\colon U_j\to V_j$ be the corresponding submanifold chart for $M$, where $U_j\coloneqq\widetilde{U_j}\cap M$ and $V_j\coloneqq\varphi_j(U_j)\subseteq P_j$ is an open subset of a polytope $P_j\subseteq E$ with non-empty interior. Let $W_j\subseteq \widetilde{V_j}$ be an open subset with $W_j\cap P_j=V_j$ and $S_j\subseteq E$ be an open convex $\varphi(j)$-neighbourhood with $\overline{S_j}\subseteq W_j$. Writing $B_j\coloneqq S_j\cap P_j$, we find a compact subset $L_j\subseteq B_j$ such that $\varphi(j)$ is in the interior of $L_j$ relative $P_j$. Then the subset $K_j\coloneqq\varphi_j^{-1}(L_j)\subseteq M$ is a compact $j$-neighbourhood in $M$.
Furthermore, there exists a smooth function $\chi_j\colon P_j\to[0,1]$ such that $\text{supp}(\chi_j)\subseteq B_j$ and $\chi_j\vert_{A_j}=1$ for some relative open subset $A_j\subseteq P_j$ with $L_j\subseteq A_j$ (see \cite[Exercise 3.5.13]{GN}). For all $k\in\N\cup\{\infty\}$ and $\sigma\in C^k(V_j,E)$, we define
\begin{align}\label{gl.t}
\Xi(\sigma)\colon P_j\to E,\quad x\mapsto
\begin{cases}
\chi_j(x)\sigma(x), &\text{ if }x\in B_j\\
0, &\text{ if }x\in P_j\setminus\text{supp}(\chi_j)
\end{cases}\end{align}
and get a
continuous linear extension operator
\[\Xi\colon C^k(V_j,E)\to C^k(P_j,E),\quad \sigma\mapsto\Xi(\sigma)\]
such that $\Xi(C_{\text{str}}^k(V_j,E))\subseteq C_{\text{str}}^k(P_j,E)$.
By Lemma \ref{ff} (a), the continuous linear map
\[\mathcal{V}^k(M)\to C^k(V_j,E),\quad \tau\mapsto d\varphi_j\circ\tau\circ\varphi_j^{-1}´\vert_{V_j}\]
takes $\mathcal{V}^k_{\text{str}}(M)$ into $C_{\text{str}}^k(V_j,E)$. 
The composition
\[\mathcal{V}^k(M)\to C^k(P_j,E),\quad \tau\mapsto \tau_j\coloneqq\Xi(d\varphi_j\circ\tau\circ\varphi_j^{-1}\vert_{V_j})\]
is continuous linear and thus smooth, and $\tau_j\in C_{\text{str}}^k(P_j,E)$ for all $\tau\in \mathcal{V}^k_{\text{str}}(M)$. To find an appropriate spray on $E$ as in Proposition \ref{m}, we choose
smooth functions $h_{j,1},h_{j,2}\colon E\to [0,1]$ such that $h_{j,1}(x)+h_{j,2}(x)=1$ for all $x\in E$, $\text{supp}(h_{j,1})\subseteq W_j$ and $\text{supp}(h_{j,2})\subseteq E\setminus\overline{S_j}$.
The condition $X(T\partial_i M)\subseteq T(T\partial_i M)$ implies that the spray 
\[X_{\varphi_j}\coloneqq(T^2 \varphi_j\circ X\circ T\varphi_j^{-1})\colon T\widetilde{V_j}\to T(T\widetilde{V_j})\]
satisfies $X_{\varphi_j}(T(V_j\cap\partial_i P_j))\subseteq T(T(V_j\cap\partial_i P_j))$ for all $i\in\{0,\ldots,n\}$. Let $X_{\text{st}}\colon TE\to T(TE),\ (x,v)\mapsto (x,v,v,0)$ be the standard spray on $E$. We define the spray
\[X_j\colon TE\to T(TE),\quad(x,v)\mapsto h_{j,1}(x)X_{\varphi_j}(x,v)+h_{j,2}(x)X_{\text{st}}(x,v),\]
where $X_{\varphi_j}(x,v)$ should be read as $0$ if $x\notin \widetilde{V_j}$. 
Then $X_j\vert_{TS_j}=X_{\varphi_j}\vert_{TS_j}$ and  $X_j(T\partial_iP)\subseteq T(T\partial_iP)$ for all $i\in\{0,\ldots,n\}$. Let $\Sigma_j\colon D_j\to E$ be the associated local addition for $E$. We apply Proposition \ref{m} to get an open $0$-neighbourhood $Q_j\subseteq C^1(P_j,E)$ such that $\Sigma_j\circ(\text{id}_{P_j},\tau)$ is well-defined for all $\tau\in Q_j$ and the smooth map
\begin{equation*}
\Psi_j\colon Q_j\to C^1(P_j,E),\quad \sigma\mapsto\Sigma_j\circ(\text{id}_{P_j},\sigma)\end{equation*}
satisfies
\begin{equation}\label{gl.f}
\Psi_j(Q_j\cap C^k_{\text{str}}(P_j,E))=\Psi_j(Q_j)\cap\text{Diff}^k(P_j)\quad\text{for all }k\in\N\cup\{\infty\}.
\end{equation}
Since $M$ is compact, there is a finite subset $J\subseteq M$ such that \[M=\bigcup_{j\in J}(K_j)^{\circ},\] where $(K_j)^{\circ}$ is the interior of $K_j$ relative $M$. We define
\[Q\coloneqq\{\tau\in\mathcal{V}^1(M)\colon\tau(M)\subseteq D\ \text{and } \tau_j\in Q_j\text{ for all }j\in J\},\]
which is an open $0$-neighbourhood in $\mathcal{V}^1(M)$. After replacing $Q$ by a $[0,1]$-saturated $C^1$-open $0$-neighbourhood contained in $Q$, we may assume throughout this proof that $Q$ is $[0,1]$-saturated.
Since the maps \[\Psi\colon Q\to C^1(M,\widetilde{M}),\quad \tau\mapsto\Sigma\circ\tau\] and
\[Q\to C^1(P_j,E),\quad \tau\mapsto\Sigma_j\circ(\text{id}_{P_j},\tau_j)\]
are smooth and thus continuous, after shrinking $Q$, we may assume that
\begin{equation}\label{gl.g}
(\Sigma\circ\tau)(\varphi_j^{-1}(\overline{B_j}))\subseteq\widetilde{U_j}
\end{equation}
and
\begin{equation}\label{gl.h}
(\Sigma_j\circ(\text{id}_{P_j},\tau_j))(\overline{A_j})\subseteq B_j
\end{equation}
for all $\tau\in Q$ and $j\in J$. Furthermore, we have $\tau_j\in Q_j\cap C^k_{\text{str}}(P_j,E)$ for all $\tau\in Q\cap\mathcal{V}_{\text{str}}^k(M)$, hence $\Sigma_j\circ(\text{id}_{P_j},\tau_j)\in\text{Diff}^k(P_j)$ by \eqref{gl.f}.
Since $\eta_{k,l}\colon\text{Diff}^{k+l}(P)\to\text{Diff}^k(P),\ f\mapsto f^{-1}$
is $C^l$ for all $l\in\N_0\cup\{\infty\}$ by Lemma \ref{aa}, it follows that the map
\begin{align}\label{gl.w}
F^j_{k,l}\colon Q\cap\mathcal{V}^{k+l}_{\text{str}}(M)\to \text{Diff}^k(P_j),\quad \tau\mapsto (\Sigma_j\circ(\text{id}_{P_j},\tau_j))^{-1}
\end{align}
is $C^l$. The subset $\lfloor L_j,A_j \rfloor\coloneqq\{f\in\text{Diff}^1(P_j)\colon f(L_j)\subseteq A_j\}$ being an open $\text{id}_{P_j}$-neighbourhood in $\text{Diff}^1(P_j)$, we see that $(F^j_{1,0})^{-1}(\lfloor L_j,A_j \rfloor)$ is an open $0$-neighbourhood in $Q\cap\mathcal{V}^1_{\text{str}}(M)$. As a consequence, there is a $C^1$-open $0$-neighbourhood $R_j\subseteq Q$ such that $R_j\cap \mathcal{V}^1_{\text{str}}(M)=(F^j_{1,0})^{-1}(\lfloor L_j,A_j \rfloor)$. If we replace $Q$ by $\bigcap_{j\in J}R_j$, we may assume that
\begin{equation}\label{gl.i}
(\Sigma_j\circ(\text{id}_{P_j},\tau_j))^{-1}(L_j)\subseteq A_j
\end{equation}
for all $\tau\in Q\cap \mathcal{V}^1_{\text{str}}(M)$ and $j\in J$. 

In the following, we show that
\begin{align}\label{gl.s}
\varphi_j(\Sigma(\tau(x)))=\Sigma_j(\varphi_j(x),\tau_j(\varphi_j(x)))\in V_j
\end{align}
for all $\tau\in Q$, $j\in J$ and $x\in\varphi_j^{-1}(A_j)$.
We have $\Sigma(\tau(x))=\pi_{T\widetilde{M}}(\gamma(1))$ for the solution $\gamma\colon [0,1]\to T\widetilde{M}$ to the initial value problem
\[\begin{cases}
\dot{y}(t)=X(y(t))\\
y(0)=\tau(x).
\end{cases}\]
Using \eqref{gl.g} and the assumption that $Q$ is balanced, we see that $\Sigma(t\tau(x))\in\widetilde{U_j}$ for all $t\in[0,1]$. Lemma \ref{p} entails that $\gamma([0,1])\subseteq T\widetilde{U_j}$. According to Remark \ref{r}, the curve $T\varphi_j\circ\gamma\colon[0,1]\to\widetilde{V_j}\times E$ solves
\[\begin{cases}
\dot{y}(t)=X_{\varphi_j}(y(t))\\
y(0)=T\varphi_j(\tau(x)).
\end{cases}\]
Let $\text{pr}_1\colon E^2\to E$ and $\text{pr}_4\colon E^4\to E$ denote the projections onto the first, resp., fourth component.
Since $T\varphi_j(\tau(x))=(T\varphi_j\circ\tau\circ\varphi_j^{-1})(\varphi_j(x))=(\text{id}_{\widetilde{V_j}},\tau_j)(\varphi_j(x)),$ it follows that $\text{pr}_1\circ T\varphi_j\circ\gamma\colon[0,1]\to\widetilde{V_j}\subseteq E$ solves
\[\begin{cases}
y''(t)=(\text{pr}_4\circ X_{\varphi_j})(y(t),y'(t))\\
y'(0)=\tau_j(\varphi_j(x))\\
y(0)=\varphi_j(x)\\
\end{cases}\]
(see \ref{s} (a)).
Moreover, we have $\Sigma_j(\varphi_j(x),\tau_j(\varphi_j(x)))=\delta(1)$ for the solution $\delta\colon [0,1]\to E$ to
\[\begin{cases}
y''(t)=(\text{pr}_4\circ X_j)(y(t),y'(t))\\
y'(0)=\tau_j(\varphi_j(x))\\
y(0)=\varphi_j(x).
\end{cases}\]
But \eqref{gl.h} shows that $\delta(t)=\Sigma_j(\varphi_j(x),t\tau_j(\varphi_j(x)))\in B_j$ for all $t\in [0,1]$, hence
\[\delta''(t)=(\text{pr}_4\circ X_j)(\delta(t),\delta'(t))=(\text{pr}_4\circ X_{\varphi_j})(\delta(t),\delta'(t)),\]
which yields $\delta=\text{pr}_1\circ T\varphi_j\circ\gamma$.
We deduce that \[\varphi_j(\Sigma(\tau(x)))=\varphi_j(\pi_{T\widetilde{M}}(\gamma(1)))=\text{pr}_1(T\varphi_j(\gamma(1)))=\delta(1)=\Sigma_j(\varphi_j(x),\tau_j(\varphi_j(x))).\]
Together with $\delta(1)\in B_j\subseteq V_j$, \eqref{gl.s} follows.

Now we prove that $(\Sigma\circ\tau)(M)=M$ holds for each $\tau\in Q\cap \mathcal{V}^1_{\text{str}}(M)$. Let $x\in M$, say $x\in K_j$ for $j\in J$. By \eqref{gl.s}, we have $\varphi_j(\Sigma(\tau(x)))\in V_j$, thus $\Sigma(\tau(x))\in U_j\subseteq M$. This yields $(\Sigma\circ\tau)(M)\subseteq M$. Since $\varphi_j(x)\in L_j$, \eqref{gl.i} provides some $v\in A_j$ such that $ \varphi_j(x)=\Sigma_j(v,\tau_j(v))$. For $y\coloneqq\varphi_j^{-1}(v)$, we obtain
\[\varphi_j(\Sigma(\tau(y)))=\Sigma_j(\varphi_j(y),\tau_j(\varphi_j(y)))=\Sigma_j(v,\tau_j(v))=\varphi_j(x)\]
using \eqref{gl.s} for the first equation. Hence $x=\Sigma(\tau(y))$, which implies that $(\Sigma\circ\tau)(M)=M$.

The next step is to show that there is a $C^1$-open $0$-neighbourhood $Q'\subseteq Q$ such that $\Sigma\circ\tau$ is injective for all $\tau\in Q'\cap\mathcal{V}^1_{\text{str}}(M)$.
To derive a contradiction, we assume that such a $Q'$ does not exist. Since $\mathcal{V}^1_{\text{str}}(M)$ is metrizable by Lemma \ref{ff} (c), we find a sequence $(\tau_m)_{m\in\N}$ in $Q\cap \mathcal{V}^1_{\text{str}}(M)$ such that $\lim_{m\to\infty}\tau_m=0$ and $\Sigma\circ\tau_m$ is not injective for all $m\in\N$. Let $x_m,y_m\in M$ such that $x_m\neq y_m$ and $(\Sigma\circ\tau_m)(x_m)=(\Sigma\circ\tau_m)(y_m)$. Since $M$ is compact and metrizable (see Lemma \ref{q}), after passing to subsequences, we may assume that both $(x_m)_{m\in\N}$ and $(y_m)_{m\in\N}$ are convergent sequences with certain limits $x$ and $y$ in $M$. The evaluation map $C(M,M)\times M\to M,\ (f,x)\mapsto f(x)$ being continuous, we see that
\[x=(\Sigma\circ 0)(x)=\lim_{m\to\infty}(\Sigma\circ\tau_m)(x_m)=\lim_{m\to\infty}(\Sigma\circ\tau_m)(y_m)=y.\]
Let $j\in J$ with $x\in (K_j)^{\circ}$. Then there exists $m\in\N$ such that $x_m,y_m\in (K_j)^{\circ}$. We obtain
\begin{align*}
\Sigma_j(\varphi_j(x_m),(\tau_m)_j(\varphi_j(x_m)))&=\varphi_j(\Sigma(\tau_m(x_m)))=\varphi_j(\Sigma(\tau_m(y_m)))\\
&=\Sigma_j(\varphi_j(x_m),(\tau_m)_j(\varphi_j(y_m))).
\end{align*}
The functions $\Sigma_j\circ(\text{id}_{P_j},(\tau_m)_j)$ and $\varphi_j$ being injective, we deduce that $x_m=y_m$, a contradiction. Consequently, a $0$-neighbourhood $Q'\subseteq Q$ as above actually exists.
If we replace $Q$ by $Q'$, we may assume that $\Sigma\circ\tau$ is injective for all $\tau\in Q\cap\mathcal{V}^1_{\text{str}}(M)$.

At this point, we have achieved that the smooth function $\Sigma\circ\tau \colon M\to M$ is a bijection for all $\tau\in Q\cap\mathcal{V}^1_{\text{str}}(M)$. Given $k\in\N\cup\{\infty\}$, this holds in particular for all $\tau\in Q\cap\mathcal{V}^k_{\text{str}}(M)$. To show that $\Sigma\circ\tau$ is even a $C^k$-diffeomorphism, let $j\in J$ and $x\in (K_j)^{\circ}$. As seen before, there is $y\in\varphi_j^{-1}(A_j)$ such that $x=(\Sigma\circ\tau)(y)$, where $(\Sigma\circ\tau)(y)=(\varphi_j^{-1}\circ(\Sigma_j\circ(\text{id}_{P_j},\tau_j))\circ\varphi_j)(y).$ It follows that
\[(\Sigma\circ\tau)^{-1}(x)=y=(\varphi_j^{-1}\circ(\Sigma_j\circ(\text{id}_{P_j},\tau_j))^{-1}\circ\varphi_j)(x),\]
hence
\begin{align}\label{gl.x}
(\Sigma\circ\tau)^{-1}\vert_{(K_j)^{\circ}}=\varphi_j^{-1}\circ(\Sigma_j\circ(\text{id}_{P_j},\tau_j))^{-1}\circ\varphi_j\vert_{(K_j)^{\circ}}.
\end{align}
Knowing that the function on the right hand side is $C^k$, we deduce that $(\Sigma\circ\tau)^{-1}$ is $C^k$ and thus $\Sigma\circ\tau$ a $C^k$-diffeomorphism. We obtain that
\[\Psi(Q\cap \mathcal{V}_{\text{str}}^k(M))\subseteq\Psi(Q)\cap\text{Diff}^k(M).\]
Now we have to establish the inverse inclusion. Given $j\in J$, for each face $F\subseteq P_j$ such that $\text{algint}(F)\cap B_j\neq\emptyset$ we choose some $x_F\in \text{algint}(F)\cap B_j$. For two faces $F,G\subseteq P_j$ of same dimension with $\text{algint}(F)\cap B_j\neq\emptyset$ and $F\neq G$, we define $r_{F,G}\coloneqq\inf_{x\in G}\Vert x-x_F\Vert>0$. Let $r$ be the minimum of all such $r_{F,G}$. Fixing any norm $\Vert\cdot\Vert$ on $E$, the subset
\[B_r(\text{id}_{\overline{B_j}})\coloneqq\{f\in C^1(\overline{B_j},E)\colon \sup_{x\in \overline{B_j}}\Vert f(x)-x\Vert<r\}\]
is open in $C^1(\overline{B_j},E)$ and the map
\[Q\to C^1(\overline{B_j},E),\quad\tau\mapsto f_{\tau,j}\coloneqq\varphi_j\circ\Sigma\circ\tau\circ\varphi_j^{-1}\vert_{\overline{B_j}}\]
is continuous. After shrinking $Q$, we may therefore assume that
\begin{align}\label{gl.j}
\sup_{x\in\overline{B_j}}\Vert f_{\tau,j}(x)-x\Vert<r
\end{align}
for all $\tau\in Q$ and $j\in J$. Moreover, since the maps
\[Q\to C^1(P_j,E),\quad\tau\mapsto g_{\tau,j}\coloneqq\Xi(f_{\tau,j}-\text{id}_{\overline{B_j}})\]
with $\Xi$ as in \eqref{gl.t}
and
\[Q\to C^1(P_j,E),\quad\tau\mapsto h_{\tau,j}\coloneqq g_{\tau,j}+\text{id}_{P_j}\]
are continuous, after shrinking $Q$ again, we may assume that
\begin{align}\label{gl.k}
\Vert g_{\tau,j}\Vert_{\infty,\text{op}}<1
\end{align}
and
\begin{align}\label{gl.l}
h_{\tau,j}\in\Psi_j(Q_j)
\end{align}
for all $\tau\in Q$ and $j\in J$. Now let $\tau\in Q$ such that $\Sigma\circ\tau\in\text{Diff}^k(M)$. Then $\tau\in\mathcal{V}^k(M)$.
Our goal is to show that $\tau$ is stratified, i.e. $\tau\in\mathcal{V}^k_{\text{str}}(M)$. For all $j\in J$, the map $f_{\tau,j}\vert_{B_j}\colon B_j\to V_j$ has a relative open image in $P_j$ and is a $C^k$-diffeomorphism onto its image.
Using Lemma \ref{t} (a), we see that $f_{\tau,j}(\partial_iP_j\cap B_j)\subseteq \partial_iP_j$ for all $i\in\{0,\ldots,n\}$. We want to check that $f_{\tau,j}\vert_{B_j}-\text{id}_{B_j}$ is stratified. For this purpose, let $F\subseteq P_j$ be a face such that $\text{algint(F)}\cap B_j\neq\emptyset$ and write $i\coloneqq n-\text{dim}(F)$. The set $\text{algint(F)}\cap B_j$ is connected, because it is convex. Thus $f_{\tau,j}(\text{algint(F)}\cap B_j)\subseteq\partial_i P_j$ is connected. 
Since the connected components of $\partial_{i}P_j$ are the sets $\text{algint}(G)$ for faces $G\subseteq P$ with $\text{dim}(F)=\text{dim}(G)$, there exists such a $G$ with $f_{\tau,j}(\text{algint}(F)\cap B_j)\subseteq\text{algint}(G)$.
In the case that $F\neq G$, \eqref{gl.j} yields
\[r_{F,G}\leq\Vert f_{\tau,j}(x_F)-x_F\Vert<r\leq r_{F,G},\]
a contradiction. Hence $F=G$ and $f_{\tau,j}(\text{algint}(F)\cap B_j)\subseteq\text{algint}(F)$. Remark \ref{ii} implies that $f_{\tau,j}\vert_{B_j}-\text{id}_{B_j}$ is stratified. Consequently, $g_{\tau,j}=\Xi(f_{\tau,j}-\text{id}_{\overline{B_j}})\in C^k_{\text{str}}(P_j,E)$. We deduce from \eqref{gl.k} and Proposition \ref{n} that $h_{\tau,j}=g_{\tau,j}+\text{id}_{P_j}\in\text{Diff}^k(P_j)$. Together with \eqref{gl.l} and \eqref{gl.f}, we obtain that $h_{\tau,j}\in\Psi_j(Q_j\cap C^k_{\text{str}}(P_j,E))$, thus there exists $\sigma\in Q_j\cap C^k_{\text{str}}(P_j,E)$ such that $h_{\tau,j}=\Sigma_j\circ(\text{id}_{P_j},\sigma)$.
Using \eqref{gl.s}, it follows that
\[\Sigma_j(y,\tau_j(y))=\varphi_j(\Sigma(\tau(\varphi_j^{-1}(y))))=h_{\tau,j}(y)=\Sigma_j(y,\sigma(y))\]
and therefore $\tau_j(y)=\sigma(y)$ for all $y\in A_j$.
Let $i\in\{0,\ldots,n\}$ and $x\in\partial_iM$, say $x\in K_j$ for $j\in J$. Then $\varphi_j(x)\in\partial_iP_j$. We obtain
\begin{align*}
T\varphi_j(\tau(x))&=(T\varphi_j\circ\tau\circ\varphi_j^{-1}\vert_{V_j})(\varphi_j(x))=(\varphi_j(x),\tau_j(\varphi_j(x)))\\
&=(\varphi_j(x),\sigma(\varphi_j(x)))\in T\partial_i P_j,
\end{align*}
using that $\sigma$ is stratified.
Hence $\tau(x)\in T\partial_iM$ and $\tau\in\mathcal{V}^k_{\text{str}}(M),$ which was to show. In conclusion, we have proved that
\[\Psi(Q)\cap\text{Diff}^k(M)\subseteq \Psi(Q\cap \mathcal{V}^k_{\text{str}}(M)),\]
whence equality holds.

Recalling that \[\Phi_k\coloneqq\Psi^{-1}\vert_{\Psi(Q)\cap C^k(M,\widetilde{M})}\colon \Psi(Q)\cap C^k(M,\widetilde{M})\to Q\cap\mathcal{V}^k(M)\] is a chart for $C^k(M,\widetilde{M})$ around $\text{id}_M$, we have shown that \[\Phi_k(\Psi(Q)\cap\text{Diff}^k(M))=Q\cap \mathcal{V}^k_{\text{str}}(M),\] which means that $\Phi_k$ is adapted to $\text{Diff}^k(M)$. 
For later use, let
\begin{align}\label{gl.y}
\Theta_k\colon \Psi(Q)\cap\text{Diff}^k(M)\to Q\cap\mathcal{V}^k_{\text{str}}(M)
\end{align}
denote the corresponding submanifold chart for $\text{Diff}^k(M)$. As the next step, we verify that $\text{Diff}^k(M)\subseteq C^k(M,\widetilde{M})$ is a submanifold modeled on $\mathcal{V}^k_{\text{str}}(M)$.
For $g\in \text{Diff}^k(M),$ the map
\[R_g\colon C^1(M,\widetilde{M})\to C^1(M,\widetilde{M}),\quad f\mapsto f\circ g\]
is a $C^{\infty}$-diffeomorphism which restricts to a $C^{\infty}$-diffeomorphism $R_g\colon C^k(M,\widetilde{M})\to C^k(M,\widetilde{M})$ such that $R_g(\text{Diff}^k(M))=\text{Diff}^k(M)$ and $R_g(\text{id}_M)=g$. Hence
\begin{align*}
(R_g\circ\Psi)(Q\cap\mathcal{V}^k_{\text{str}}(M))=R_g(\Psi(Q)\cap\text{Diff}^k(M))=(R_g\circ\Psi)(Q)\cap\text{Diff}^k(M),
\end{align*}
showing that 
\[\Phi_k\circ \big(R_g\vert_{\Psi(Q)\cap C^k(M,\widetilde{M})}\big)^{-1}\
\colon (R_g\circ\Psi)(Q)\cap C^k(M,\widetilde{M})\to Q\cap \mathcal{V}^k(M)\] 
is a chart for $C^k(M,\widetilde{M})$ around $g$ adapted to $\text{Diff}^k(M)$. This completes the proof of Proposition \ref{hh}.

We use the submanifold structure on $\text{Diff}^k(M)$ to prove Proposition \ref{gg} for all $k\in\N\cup\{\infty\}$. First, we check that (b) holds, which subsumes the exponential law in (a). 
Let $m\in\N$, $\alpha_1,\ldots,\alpha_m\in\N_0\cup\{\infty\}$ and $g\colon L_1\times\ldots\times L_m\to \text{Diff}^k(M)$ be a map on a product of $C^{\alpha_i}$-manifolds $L_i$ for $i\in\{1\ldots,m\}$, possibly with rough boundary. Since $\text{Diff}^k(M)\subseteq C^k(M,\widetilde{M})$ is a submanifold, $g$ is $C^{(\alpha_1,\ldots,\alpha_m)}$ if and only if $g\colon L_1\times\ldots\times L_m\to C^k(M,\widetilde{M})$ is $C^{(\alpha_1,\ldots,\alpha_m)}$ as a map to $C^k(M,\widetilde{M})$. According to \cite[Theorem 1.1]{GS}, the latter holds if and only if $g^{\wedge}\colon L_1\times\ldots\times L_m\times M\to \widetilde{M},\ g^{\wedge}(x_1,\ldots,x_m,y)\coloneqq g(x_1,\ldots,x_m)(y)$ is $C^{(\alpha_1,\ldots,\alpha_m,k)}$, which is equivalent to $g^{\wedge}\colon L_1\times\ldots\times L_m\times M\to M$ being $C^{(\alpha_1,\ldots,\alpha_m,k)}$.

To prove the uniqueness of the manifold structure, 
we write $\text{Diff}^k(M)'$ for $\text{Diff}^k(M)$ endowed with another smooth manifold structure for which the exponential law as in (a) holds. The identity map $\text{id}\colon \text{Diff}^k(M)\to \text{Diff}^k(M)$ being smooth, the exponential law entails that the evaluation map \[\varepsilon_k=\text{id}^{\wedge}\colon \text{Diff}^k(M)\times M\to M\]
is $C^{\infty,k}$. Since the map $g\coloneqq\text{id}\colon \text{Diff}^k(M)\to \text{Diff}^k(M)'$  satisfies $g^{\wedge}=\varepsilon_k$, the exponential law shows that $g$ is smooth. Reversing the roles of $\text{Diff}^k(M)$ and $\text{Diff}^k(M)'$, it follows that also $g^{-1}$ is smooth. Thus $g$ is a $C^{\infty}$-diffeomorphism and the manifold structures $\text{Diff}^k(M)'$ and $\text{Diff}^k(M)$ coincide.

Next we show that (c) holds for all $l\in\N_0\cup\{\infty\}$.
By Lemma \ref{bb}, the map
\[c_{k,l}^{\wedge}=\varepsilon_{k+l}\circ(\text{id}\times\varepsilon_k)\colon\text{Diff}^{k+l}(M)\times\text{Diff}^k(M)\times M\to M,\quad (f,g,x)\mapsto f(g(x))\] 
is $C^{\infty,l,k}$. Hence (b) entails that the composition map
\[c_{k,l}\colon\text{Diff}^{k+l}(M)\times\text{Diff}^k(M)\to\text{Diff}^k(M),\quad (f,g)\mapsto f\circ g\]
is $C^{\infty,l}$. 
Now let us verify that the inversion map
\[\iota_{k,l}\colon\text{Diff}^{k+l}(M)\to\text{Diff}^k(M),\quad f\mapsto f^{-1}\]
is $C^l$. 
Using $F^j_{k,l}\colon Q\cap \mathcal{V}^{k+l}_{\text{str}}(M)\to\lfloor L_j, A_j\rfloor\cap\text{Diff}^k(P_j)$ defined as in \eqref{gl.w} for $j\in J$ and the chart $\Theta_{k+l}$ for $\text{Diff}^{k+l}(M)$ as in \eqref{gl.y}, we get a $C^l$-map
\[F^j_{k,l}\circ\Theta_{k+l}\colon \text{Diff}^{k+l}(M)\cap \Psi(Q)\to \lfloor L_j, A_j\rfloor\cap\text{Diff}^k(P_j).\]
The equation \eqref{gl.x} entails that
\[f^{-1}\vert_{(K_j)^{\circ}}=\varphi_j^{-1}\circ F^j_{k,l}(\Theta_{k+l}(f))\circ\varphi_j\vert_{(K_j)^{\circ}}\]
for all $f\in \text{Diff}^{k+l}(M)\cap \Psi(Q)$. The evaluation map
$\varepsilon^{j}_k\colon\text{Diff}^k(P_j)\times P_j\to P_j$
being $C^{\infty,k}$, it follows that
\begin{align*}
\varphi_j^{-1}\circ\varepsilon^{j}_k\circ((F^j_{k,l}\circ\Theta_{k+l})\times\varphi_j\vert_{(K_j)^{\circ}})\colon(\text{Diff}^{k+l}(M)\cap\Psi(Q))\times (K_j)^{\circ}&\to M,\\
(f,x)&\mapsto f^{-1}(x)
\end{align*}
is $C^{l,k}$. 
Since $M$ is covered by the subsets $(K_j)^{\circ}$ for $j\in J$, we deduce that
\[(\text{Diff}^{k+l}(M)\cap\Psi(Q))\times M\to M,\quad (f,x)\mapsto f^{-1}(x)\]
is $C^{l,k}$, whence the exponential law entails that
\[\iota_{k,l}\vert_{\Psi(Q)}\colon \text{Diff}^{k+l}(M)\cap\Psi(Q)\to \text{Diff}^k(M),\quad f\mapsto f^{-1}\]
is $C^l$. 
Given $g\in \text{Diff}^{k+l}(M)$, we compose the $C^l$-map
\[L_{g^{-1}}\colon \text{Diff}^k(M)\to \text{Diff}^k(M),\quad f\mapsto g^{-1}\circ f\]
with $\iota_{k,l}\vert_{\Psi(Q)}$ and the $C^{\infty}$-map
\[R_{g^{-1}}\colon\text{Diff}^{k+l}(M)\cap R_g(\Psi(Q))\to \text{Diff}^{k+l}(M)\cap\Psi(Q),\quad f\mapsto f\circ g^{-1},\]
which yields the $C^l$-map
\[
L_{g^{-1}}\circ\iota_{k,l}\vert_{\Psi(Q)}\circ R_{g^{-1}}\colon \text{Diff}^{k+l}(M)\cap R_g(\Psi(Q))\to \text{Diff}^k(M),\quad f\mapsto f^{-1},\]
where $\text{Diff}^{k+l}(M)\cap R_{g}(\Psi(Q))$ is an open $g$-neighbourhood in $\text{Diff}^{k+l}(M)$. As $g$ was arbitrary, it follows that $\iota_{k,l}\colon\text{Diff}^{k+l}(M)\to\text{Diff}^k(M)$ is $C^l$, hence (c) holds.

Choosing $l=0$ in (c), we see that the group multiplication $c_{k,0}\colon\text{Diff}^k(M)\times\text{Diff}^k(M)\to \text{Diff}^k(M)$ and the group inversion $\iota_{k,0}\colon\text{Diff}^k(M)\to \text{Diff}^k(M)$ are continuous, thus $\text{Diff}^k(M)$ is a topological group. Taking $k=l=\infty$, we obtain that $c_{\infty,\infty}\colon \text{Diff}^{\infty}(M)\times\text{Diff}^{\infty}(M)\to \text{Diff}^{\infty}(M)$ and $\iota_{\infty,\infty}\colon\text{Diff}^{\infty}(M)\to \text{Diff}^{\infty}(M)$ are smooth, entailing that $\text{Diff}^{\infty}(M)$ is a Lie group.

The results obtained here imply that $\text{Diff}^{\infty}(M)$
is $L^1$-regular in the sense of \cite{G15}. In fact,
a forthcoming version of \cite{GS24}
shows that a Fr\'{e}chet-Lie group $G$ is $L^1$-regular
whenever $G=\bigcap_{k\in {\mathbb N}}G_k$ with smooth Banach manifolds $G_1\supseteq G_2\supseteq \ldots$ which are groups, such that the following conditions are satisfied:
\begin{itemize}
\item[(a)] For each $k\in {\mathbb N}$, the inclusion map $G_{k+1}\to G_k$ is a smooth group homomorphism;
\item[(b)] The group multiplication $G_{k+l}\times G_k\to G_k$ is $C^l$ for
all $k\in {\mathbb N}$ and $l\in {\mathbb N}_0$;
\item[(c)] The modelling spaces $E_k$ of $G_k$
form a descending sequence $E_1\supseteq E_2\supseteq\ldots$
with continuous linear inclusion maps $E_{k+1}\to E_k$;
\item[(d)] The modelling space of $G$ is $F=\bigcap_{k\in {\mathbb N}}E_k$ as a vector space and $F$ has the initial topology with respect to
the inclusion maps $F\to E_k$, making it the projective limit of the spaces $E_k$;
\item[(e)] There exist charts $\phi_k\colon U_k\to V_k$
from an open $e$-neighbourhood $U_k\subseteq G_k$ onto an open
$0$-neighbourhood $V_k\subseteq E_k$ such that
$U_k=U_1\cap G_k$, $V_k=\phi_1(U_k)$ and $\phi_k=\phi_1|_{U_k}^{V_k}$
for each $k\in {\mathbb N}$; and
\item[(f)] $U:=\bigcap_{k\in {\mathbb N}}U_k$ is open in $G$,
$V:=\bigcap_{k\in \N}V_k$ is open in $F$ and $\phi_1\vert_U^V$ is a chart of $G$.
\end{itemize}
These conditions are satisfied by $G:=\text{Diff}^{\infty}(M)$
if we take $G_k:=\text{Diff}^k(M)$ and $\phi_k\coloneqq\Theta_k$ as defined in \eqref{gl.y}. 
\end{proof}



\newpage

\addcontentsline{toc}{section}{\protect\numberline{}Literatur}

Johanna Jakob, Universität Paderborn, Institut für Mathematik, Warburger Str. 100, 33098 Paderborn, Germany


\begin{thebibliography}{99}

\bibitem{Alzaareer} Alzaareer, H., \emph{Differential calculus on multiple products}, Indag. Math. \textbf{30} (2019), 1036–1060.

\bibitem{AS} Alzaareer H. and A. Schmeding, \emph{Differentiable mappings on products with different degrees of differentiability in the two factors}, Expo. Math. \textbf{33} (2015), 184–222. 

\bibitem{Amiri} Amiri, H., H. Glöckner and A. Schmeding, \emph{Lie groupoids of mappings taking values in a Lie groupoid}, Arch. Math. (Brno) \textbf{56} (2020), 307–356.

\bibitem{Bastiani} Bastiani, A., Applications diﬀ\'erentiables et vari\'et\'es diﬀ\'erentiables de dimension inﬁnie, J. Anal. Math. \textbf{13} (1964), 1–114.

\bibitem{Brondsted} Brøndsted, A., \emph{An Introduction to Convex Polytopes}, Springer, New York, 1983.

\bibitem{Eells} Eells, J. Jr., \emph{A setting for global analysis}, Bull. Amer. Math. Soc. \textbf{72} (1966), 751–807.

\bibitem{G} Glöckner, H., \emph{Diffeomorphism groups of convex polytopes}, J. Convex Anal. \textbf{30} (2023), 343–358.

\bibitem{G15} Glöckner, H., \emph{Measurable regularity properties of infinite-dimensional Lie groups}, preprint, 2015, arxiv:1601.02568.

\bibitem{GN} Glöckner, H. and K.-H. Neeb, “Inﬁnite Dimensional Lie Groups”, book in preparation, 2025.

\bibitem{GS} Glöckner, H. and A. Schmeding, \emph{Manifolds of mappings on cartesian pro\-ducts}, preprint, 2021, arXiv:2109.01804.

\bibitem{GS24} Glöckner, H. and A. Suri, \emph{$L^1$-regularity of strong ILB-groups}, preprint, 2024, arXiv:2410.02909.

\bibitem{Grong} Grong, E. and A. Schmeding, \emph{Controllability and diffeomorphism groups on manifolds with boundary}, Forum Math. \textbf{37} (2025), 1291–1308.

\bibitem{HB} Hermas, N. and N. Bedida, \emph{On a class of regular Fr\'echet-Lie groups}, Bull. Iran. Math. Soc. \textbf{47} (2021), 627–647.

\bibitem{J} Jakob,  J., \emph{Der Whitneysche Fortsetzungssatz für vektorwertige Funktionen}, preprint, 2023, arXiv:2307.03473.

\bibitem{Kriegl} Kriegl, A. and P. W.  Michor, “The Convenient Setting of Global Analysis”, AMS, Providence, 1997a.

\bibitem{Michor} Michor, P. W., “Manifolds of Differentiable Mappings,” Shiva Publ., Nantwich, 1980.

\bibitem{Milnor} Milnor, J., \emph{Remarks on inﬁnite-dimensional Lie groups}, pp. 1007–1057 in: B. S. DeWitt and R. Stora, “Relativit\'e, groupes et topologie II”, North Holland, Amsterdam, 1984.


\bibitem{Schmeding} Schmeding, A., “An Introduction to Infinite-Dimensional Differential Geometry,” Cambridge University Press, Cambridge, 2023.

\end{thebibliography}
\end{document}